\documentclass[a4paper,10pt]{amsart}

\usepackage{latexsym}
\usepackage{amsmath,amssymb,amsxtra,amscd,amsthm}
\usepackage{url}
\usepackage[breaklinks]{hyperref}

\def\cX{\mathcal{X}}
\def\mC{\mathbb{C}}
\def\mN{\mathbb{N}}
\def\mP{\mathbb{P}}
\def\mZ{\mathbb{Z}}

\newcommand{\diff}[2]{\mathrm{d}^{#1}{#2}}
\newcommand{\e}[1]{\mathrm{e}^{#1}}
\newcommand{\Hyp}[3]{\,{}_3F_2\!\left(\genfrac{}{}{0pt}{0}{#1}{#2} \,; #3\right)}
\DeclareMathOperator{\Res}{Res}

\theoremstyle{break}
\newtheorem{defin}{Definition}[section]
\newtheorem{prop}[defin]{Proposition}
\newtheorem{cor}[defin]{Corollary}
\newtheorem{lem}[defin]{Lemma}
\newtheorem{thm}[defin]{Theorem}
\newtheorem*{thm64}{Theorem 6.5}
\newtheorem*{thm71}{Theorem 7.1}
\newtheorem*{thm72}{Theorem 7.2}
\theoremstyle{definition}
\newtheorem{rem}[defin]{Remark}
\newtheorem{expl}[defin]{Example}

\numberwithin{equation}{section}

\allowdisplaybreaks

\begin{document}

\title[Analytic Continuation of Hypergeometric Functions]{Analytic Continuation of Hypergeometric Functions in the Resonant Case}


\author{Emanuel Scheidegger}
\address{Mathematisches Institut\\Albert--Ludwig--Universit\"at Freiburg}
\curraddr{}
\email{emanuel.scheidegger@math.uni-freiburg,de}
\thanks{}


\subjclass[2000]{Primary }

\date{\today}

\begin{abstract}
  We perform the analytic continuation of solutions to the
  hypergeometric differential equation of order $n$ to the third
  regular singularity, usually denoted $z=1$, with the help of
  recurrences of their Mellin--Barnes integral representations. In the
  resonant case, there are necessarily logarithmic solutions. We apply
  the result to Picard-Fuchs equations of certain one--parameter
  families of Calabi--Yau manifolds, known as the mirror quartic and
  the mirror quintic. 
\end{abstract}

\maketitle

\section{Introduction}
\label{sec:introduction}

Although monodromy and analytic continuation properties of the
solutions of the hypergeometric differential equation have been
well--studied in the past, interest got revived by mirror
symmetry~\cite{Candelas:1991rm,Batyrev:1995wa}. In particular,
Picard--Fuchs equations for one--parameter families of Calabi--Yau
manifolds of dimension $n-1$ whose Gauss--Manin connection has three
regular singular points, one of which is of maximal unipotent
monodromy, is a hypergeometric differential equation of order $n$ with
resonant exponents. In the case $n=4$, such differential equations
have been classified in~\cite{Doran:2006ab}. 

The three regular singular points are usually taken to be
$z=0,1,\infty$. The analytic continuation of the solutions from $z=0$ to
$z=\infty$ is well--known due to the fact that the differential equation
after transformation to $w=\frac{1}{z}$ is again of hypergeometric
form. The analytic continuation of the solutions from $z=0$ to the
$z=1$ is more difficult to obtain since the differential equation after the
transformation to $y=1-z$ is not of hypergeometric form for
$n>2$. Nevertheless, it is known in the nonresonant case, i.e. if
there are no logarithmic solutions at
$0$~\cite{Norlund:1955ab,Buehring:1992ab}. We present here a complete
solution to the problem of analytic continuation to $z=1$ in the
resonant case. Before we discuss the solution, let us present two
applications of interest in the context of mirror
symmetry. Let $\Phi_{0}(z),\Phi_{1}(z)$ denote certain chosen fundamental
matrices of the hypergeometric differential equation near $z=0,1$.   
The variation of polarized Hodge structure of the family $\pi: \cX \to
\mP^1$ of mirror quartics given as
\[
  \cX_z = \{{x_0}^4 + {x_1}^4 + {x_2}^4 + {x_3}^4 -
  4\,z^{-\frac{1}{4}}x_0x_1x_2x_3 = 0\} \subset \mP^3 
\]
leads to a hypergeometric differential equation of order 3. Then we
prove
\begin{thm71}
  The analytic continuation of $\Phi_0(z)$ to $z=1$ is determined by
  $\Phi_0(z) = \Phi_1(1-z) M_{10}$ with
  \[
    M_{10} =
    \begin{pmatrix}
      \frac{A}{2\,\sqrt{2}\pi} & -\frac{A}{4\pi i}  & 0\\
      \frac{2}{\sqrt{2}\pi}
      \left(\frac{3\,A}{64} + \frac{1}{A}\right) &  -\frac{1}{\pi i}
      \left(\frac{3\,A}{64} - \frac{1}{A}\right) & 0 \\
      -\frac{2}{\sqrt{2}\pi} & 0 & -\frac{1}{\sqrt{2}\pi}
    \end{pmatrix}
  \]
  where $A=\frac{\Gamma(\frac{1}{8})
    \Gamma(\frac{3}{8})}{\Gamma(\frac{5}{8}) \Gamma(\frac{7}{8})}$. 
\end{thm71}
Similarly, the variation of polarized Hodge structure of the family $\pi:\cX \to
\mP^1$ of mirror quintics given as
\[
  \cX_z = \{{x_0}^5 + {x_1}^5 + {x_2}^5 + {x_3}^5 + {x_4}^5 -
  5\,z^{-\frac{1}{5}}x_0x_1x_2x_3x_4 = 0\} \subset \mP^4 
\]
leads to a hypergeometric differential equation of order 4. Then, we prove
\begin{thm72}
  Let $\Phi_{z_0}(z)$ be the fundamental matrices near $z_0=0,1$.
  The analytic continuation of $\Phi_0(z)$ to $z=1$ is determined by
  $\Phi_0(z) = \Phi_1(1-z) M_{10}$ with
  \[
    M_{10} =
    \begin{pmatrix}
     \medskip
      l_0 & -\frac{h_0}{2\pi i} & \frac{5\,k_0}{(2\pi i)^2} & 0\\
     \medskip
     w_0 & -\frac{h_1}{2\pi i} &\frac{5\,k_1}{(2\pi i)^2} & 2\pi i \\
     \medskip
     1 & 0 & 0 & 0 \\
     w_1 - \frac{7}{10}w_0 & -\frac{h_2}{2\pi i} + \frac{7}{10}\frac{h_1}{2\pi i}
     &\frac{5\,k_2}{(2\pi i)^2} - \frac{7}{10}\frac{5\,k_1}{(2\pi
       i)^2} & 0
    \end{pmatrix}
  \]
\end{thm72}
  The real constants $l_0,w_0,w_1,h_0,h_1,h_2$ and the complex
  constants $k_0,k_1,k_2$ have an
  analytic expression which is not very simple, and therefore are
  given in the main text 
  in~\eqref{eq:lqw}, \eqref{eq:hm} and~\eqref{eq:km2}. This result
  might be useful to describe the embedding of the parameter space
  $\mP^1$ into the period domain of the family $\cX$. 

We summarize how these results are obtained.
In Section~\ref{sec:hyperg-funct} we briefly review the known bases of
solutions to the hypergeometric differential equations 
following~\cite{Norlund:1955ab} in terms of integral representations
of the Mellin--Barnes type. We define the notion of resonant
exponents. In the resonant case, we introduce a new basis of 
solutions $G_p(z)$ defined in terms of a Mellin--Barnes integral as
\[
  G_p(z) := \int \frac{\diff{}{t}}{2\pi i} \e{i\pi(p-2)t} z^t \prod_{j=1}^n
  \frac{\Gamma(\alpha_j+t)}{\Gamma(1-\gamma_j+t)} \prod_{h=1}^p \Gamma(\gamma_h-t)\Gamma(1-\gamma_h+t)
\]
These solutions will be shown to have the following straightforward
power series expansion at $z=1$. 
\begin{thm64}
   For any $2 < p \leq n$, if $|z-1| < 1$, $\Re \beta_n > \Re\beta_p$,
 $\Re(\alpha_s+\gamma_j)>0$, $j=1\dots,p$, $s=p+1,\dots,n$, $\alpha_{p}+\gamma_p,
    \alpha_s+\gamma_{s+1} \not \in \mZ_{\leq 0}$, $s=2,\dots,p-1$ then
  \[
    \begin{aligned}
    G_p(z) &= \sum_{m=0}^\infty \Gamma(\alpha_1+\gamma_2) \int \frac{\diff{}{v}}{2\pi i} \e{-i\pi v} \Gamma(\alpha_1+\gamma_1+v) \Gamma(-v) \\
   &\phantom{=} \cdot \int \frac{\diff{}{s}}{2 \pi i} \frac{B_{p,m}(s)}{\Gamma(m+1)} \frac{ \Gamma(\gamma_2-s) \Gamma(\gamma_1+v-s) }{\Gamma(\alpha_1+\gamma_1+\gamma_2+v-s)}\\
   &\phantom{=} \cdot \int
    \frac{\diff{}{u}}{2\pi i} \e{-i\pi u} 
    \frac{\Gamma(-v+u)\Gamma(-u)}{\Gamma(-v)} \prod_{s=p+1}^n
    \frac{\Gamma(\alpha_s+\gamma_1+u)}{\Gamma(1-\gamma_s+\gamma_1+u)} (1-z)^m
  \end{aligned}
  \]
  If $p=2$ then
  \[
  \begin{aligned}
    G_2(z) &= \sum_{m=0}^\infty
    \frac{\Gamma(\alpha_1+\gamma_2+m)\Gamma(\alpha_2+\gamma_2+m)}{\Gamma(m+1)}
    \\
   &\phantom{=} \cdot \int \frac{\diff{}{v}}{2\pi i} \e{-i\pi v} \frac{\Gamma(\alpha_1+\gamma_1+v) \Gamma(\alpha_2+\gamma_1+v)  \Gamma(-v)}{\Gamma(\alpha_1+\alpha_2+\gamma_1+\gamma_2+m+v)} \\
   &\phantom{=} \cdot \int
    \frac{\diff{}{u}}{2\pi i} \e{-i\pi u} 
    \frac{\Gamma(-v+u)\Gamma(-u)}{\Gamma(-v)} \prod_{s=3}^n
    \frac{\Gamma(\alpha_s+\gamma_1+u)}{\Gamma(1-\gamma_s+\gamma_1+u)} (1-z)^m
  \end{aligned}
  \]
\end{thm64}
The functions $B_{p,m}(s)$ are given in terms of multiple
Mellin--Barnes integrals.

The strategy to prove these statements is to find recurrences for all
the solutions. The Mellin--Barnes integral representation of a
solution to an order $n$ equation is written in terms of a
Mellin--Barnes integral representation of a solution to an order $n-1$
equation. In this way, the problem of analytic continuation is reduced
to the one of solutions of an order $2$ equation. The latter is well--known and
briefly reviewed in Section~\ref{sec:class-hyperg-funct}. In
Section~\ref{sec:recurr-relat-order} 
we collect the known recurrences due to~\cite{Norlund:1955ab}
and~\cite{Buehring:1992ab}. We reformulate the latter in terms of
integral representations and prove a new recurrence for the solutions
$G_p(z)$. Equipped with these results, we perform the analytic
continuation on the level of integral representations in
Section~\ref{sec:analyt-cont-z=1}. In order to make contact to the
bases of solutions obtained from the 
Frobenius method in terms of power series expansions, we expand in
Section~\ref{sec:series-expansions} the integral representations
obtained in Section~\ref{sec:analyt-cont-z=1} in powers of
$1-z$. Finally, in Section~\ref{sec:applications} we apply these
series expansions to the examples presented above, the
Picard--Fuchs equations of the family of mirror quartics and of the
family of mirror quintics.
 
We wish to emphasize that in Sections~\ref{sec:hyperg-funct}
to~\ref{sec:analyt-cont-z=1} all results are formulated in terms of
integral representations. At no point, such an 
integral is evaluated and converted into a power series. This will
only be done in Section~\ref{sec:series-expansions} in order to be able to discuss
the explicit examples afterwards. For this reason, and also to make
the presentation self--contained we include proofs of some of the results due
to~\cite{Norlund:1955ab} and~\cite{Buehring:1992ab}. 

\subsection*{Acknowledgments}
\label{sec:acknowledgments}

The author is very grateful to Johanna Knapp for invaluable help with
checking the results of Section 7 and comments on the manuscript, and
to Don Zagier for fruitful discussions. Moreover, he profited from a
collaboration with Johanna Knapp and Mauricio Romo on a closely
related physics project.

\section{The hypergeometric equation}
\label{sec:hyperg-funct}

\subsection{The hypergeometric differential equation and its solutions}
\label{sec:hyperg-diff-equat}

In this subsection we review the properties of the hypergeometric differential equation and its solutions that will be needed later on. For more background see e.g.~\cite{Erdelyi:1953ab,Slater:1966ab,Heckman:2015ab}. We will, however, follow the notation of N\o rlund~\cite{Norlund:1955ab}. 

We consider the hypergeometric differential equation of order $n$
\begin{equation}
  \left( \theta\, \prod_{j=1}^{n-1} \left( \theta - \gamma_j\right) -
    z \prod_{j=1}^n \left(\theta - \alpha_j\right) \right) y(z) =
  0 
  \label{eq:DE}
\end{equation}
where $\alpha_1,\dots,\alpha_n,\gamma_1,\dots,\gamma_{n-1} \in \mathbb{C}$ and $\theta = z\frac{\diff{}{}}{\diff{}{z}}$. For convenience, to make subsequent expressions more symmetric, we also introduce $\gamma_n$ and set $\gamma_n=0$. 
This differential equation has regular singularities at $z_0=0,1,\infty$ with exponents
\[
    \begin{array}[]{c|cccccc}
    z_0= 0 & \gamma_1 & \gamma_2 & \gamma_3 & \dots & \gamma_{n-1} & 0\\
    z_0= 1 & 0 & 1 & 2 & \dots & n-2 & \beta_n\\
    z_0 = \infty & \alpha_1 & \alpha_2 & \alpha_3 & \dots & \alpha_{n-1}
    & \alpha_n
  \end{array}
\]
where we have set
$$
  \beta_n = n - 1 - \sum_{i=1}^n (\alpha_i + \gamma_i).
$$
Furthermore we define
$$
  f_n(t) = \prod_{j=1}^n
  \frac{\Gamma(\alpha_j+t)}{\Gamma(1-\gamma_j+t)} 
$$

\begin{defin}
  Let $q\in \{2,\dots,n\}$ and $E$ be a maximal subset of exponents
  $\{\gamma_{i_1},\dots,\gamma_{i_q}\}$ at
  $z_0=0$, or of exponents $\{\alpha_{i_1},\dots,\alpha_{i_q}\}$ at
  $z_0=\infty$, with the property that $\lambda
  - \mu \in \mZ$ for all $\lambda,\mu \in E$. Then $E$ is called to be
  in resonance or resonant. If there is no resonant subset of a set of
  exponents, we call the set of exponents nonresonant.
\end{defin}

\begin{rem}
  \begin{enumerate}
  \item If $\gamma_1,\dots,\gamma_q$ are resonant, N\o rlund uses the
    terminology that they ``form a group''. Since they do not -- mathematically
    speaking -- form a group, we prefer the notion of resonance which
    is well known in the context of GKZ hypergeometric
    systems~\cite{Gelfand:1989ab,Stienstra:2007ab}. The latter are a 
    natural generalization of the hypergeometric differential equation
    discussed here.
  \item The differential equation~\eqref{eq:DE} remains invariant if
    we interchange $\alpha_j$ and $\gamma_j$, $j=1,\dots,n$ and
    replace $z$ by $\frac{1}{z}$. For simplicity, we will therefore
    restrict ourselves to resonant subsets of $\{\gamma_1,\dots,\gamma_n\}$.
  \item In the following we only deal with the case that there is a
    single subset $E$ of $\{\gamma_1,\dots,\gamma_n\}$ which is in
    resonance. We reorder the exponents such that
    $E=\{\gamma_1,\dots,\gamma_q\}$. If $\{\gamma_1,\dots,\gamma_n\}$
    contains more than one resonant subset, the statements in this article
    are to be applied to each subset seperately.   
  \item If $z=0$ is a point of maximal unipotent monodromy, then
    $\gamma_i=0$, $i=1,\dots,n$, and hence
    $E=\{\gamma_1,\dots,\gamma_n\}$. 
  \end{enumerate}
\end{rem}

Next, we discuss bases of solutions to~\eqref{eq:DE} near each
singularity in terms of Mellin--Barnes integrals following N\o
rlund. 

\begin{itemize}
\item Near $z_0=0$: 

If $\gamma_1,\dots,\gamma_n$ are not in resonance, N\o rlund gives the
basis 
\begin{equation}
  \label{eq:nonresonantyj}
  \begin{aligned}
    y_j^*(z) &:= f_n(\gamma_j) z^{\gamma_j} {}_n F_{n-1}
    \left(\genfrac{}{}{0pt}{}{\alpha_1+\gamma_j,\dots,
      \alpha_n+\gamma_j}{1-\gamma_1+\gamma_j, \widehat{\dots}, 1-\gamma_{n}+\gamma_j};z\right) \\
    &:= \int \frac{\diff{}{t}}{2\pi i}\, \e{-i\pi t} z^t f_n(t)
    \Gamma(\gamma_j-t)\Gamma(1-\gamma_j+t), \qquad j=1,\dots,n
  \end{aligned}
\end{equation}
for $|\arg(-z)| < \pi$. The $\widehat{\phantom{x}}$ denotes the omission of the
parameter $1-\gamma_j+\gamma_j$. 

If $\gamma_1, \dots, \gamma_q$, $2\leq q \leq n$, are in resonance,
then the solutions $y_j^*(z)$ are all equal for $1\leq j \leq q$ and
have to be replaced by
\begin{equation}
  \label{eq:resonantyj}
  y_j^*(z) := \int \diff{}{t} \, z^t f_n(t) \frac{1}{\left(1-\e{2\pi
      i(t-\gamma_1)}\right)^j} \qquad j=1,\dots,q
\end{equation}
for $0 < \arg(z) < 2\pi j$.
N\o rlund shows that after this replacement for each resonant subset,
the functions $y_j^*(z)$ form a linearly independent set of
solutions. Note that $y_1^*(z)$ is the same for both
cases. Furthermore, if we evaluate the integrals using the residue
theorem, we obtain series expansion containing polynomials in $\log
z$. Hence, we will refer to these solutions as the logarithmic
solutions. 

\item Near $z_0=\infty$: 

As already mentioned, the differential equation~\eqref{eq:DE} remains invariant if
    we interchange $\alpha_j$ and $\gamma_j$, $j=1,\dots,n$ and
    replace $z$ by $\frac{1}{z}$. A basis of solutions is given by
\begin{equation}
  \label{eq:nonresonantybarj}
  \begin{aligned}
    \overline{y}_j^*(z) &:= \prod_{k=1}^n
  \frac{\Gamma(\gamma_k+\alpha_j)}{\Gamma(1-\alpha_k+\alpha_j)}  z^{\alpha_j} {}_n F_{n-1}
    \left(\genfrac{}{}{0pt}{}{\gamma_1+\alpha_j,\dots,
      \gamma_n+\alpha_j}{1-\alpha_1+\alpha_j, \widehat{\dots}, 1-\alpha_{n}+\alpha_j};z\right) \\
    &:= \int \frac{\diff{}{t}}{2\pi i}\, \e{-i\pi t} z^t \prod_{k=1}^n
  \frac{\Gamma(\gamma_k+t)}{\Gamma(1-\alpha_k+t)} 
    \Gamma(\alpha_j-t)\Gamma(1-\alpha_j+t),  \end{aligned}
\end{equation}
for $j=1,\dots,n$ and $|\arg(-z)| < \pi$.

\item Near $z_0=1$: 

If $\gamma_1,\dots,\gamma_n$ are not in resonance, a basis
of solutions can be described as follows. There is the special
solution corresponding to the exponent $\beta_n$. N\o rlund denotes
this solution $\xi_n(z)$ if $\beta_n \not \in \mZ_{<0}$ and
$\eta_n(z)$ otherwise. We have
\[
\begin{aligned}
  \xi_n(z) &:= \xi_n\left(\genfrac{}{}{0pt}{}{\alpha_1,\dots,
      \alpha_n}{\gamma_1, \dots, \gamma_{n}};z\right) := z^{\gamma_1}
  (1-z)^{\beta_n} \sum_{k=0}^\infty c_k\,(1-z)^k\\
  \eta_n(z) &:= z^{\gamma_1} \sum_{k=0}^\infty
  \frac{c_{k-\beta_n}}{\Gamma(k+1)} \,(1-z)^k
\end{aligned}
\]
omitting the parameters when no confusion is possible.  The
coefficients $c_k$ are determined recursively by the differential
operator. Explicit formulas can be found in~\cite{Norlund:1955ab}. If
in $\xi_n(z)$ we interchange $\alpha_j$ and $\gamma_j$, $j=1,\dots,n$,
and replace $z$ by $\frac{1}{z}$, we obtain a 
solution which we denote by $\bar \xi_n(z)$ and differs from
$\xi_n(z)$ by a factor $\e{\pm i\pi \beta_n}$, similarly for
$\eta_n(z)$. Furthermore, there are $n-1$ holomorphic solutions which
again according to N\o rlund can be taken to be (for fixed $i$)  
\[
\begin{aligned}
  y_{ij}(z) &= \frac{\pi}{\sin\pi(\gamma_j-\gamma_i)} \left( y_j^*(z)
    -
    y_i^*(z) \right)\\
  &= \int \frac{\diff{}{t}}{2\pi i} z^t f_n(t)
  \frac{\pi}{\sin\pi(\gamma_j-t)}\frac{\pi}{\sin\pi(\gamma_i-t)},
\end{aligned}
\]
for $j =1,\dots,i-1,i+1,\dots,n$ and $y_j^*(z)$ as
in~\eqref{eq:nonresonantyj}.  The integral is convergent for 
$-2\pi < \arg z < 2 \pi$ and therefore represents a solution which is
regular at $z = 1$. N\o rlund only discusses these solutions when
$\gamma_i - \gamma_j \not \in \mZ$. If $\gamma_i - \gamma_j \in \mZ$
we appearently divide 0 by 0, but the quotient is well-defined as is
easily seen from the integral representation. If there are three or
more exponents in resonance, then the corresponding $y_{ij}$ become
linearly dependent. In the next subsection, we will give a basis of
solutions in the resonant case in terms of integral representations,
as well.
\end{itemize}

\subsection{Meijer $G$-functions}
\label{sec:meijer-g-functions}

We consider a special instance of the Meijer G--function~\cite[\S 5.3]{Erdelyi:1953ab}
\[
  G^{p,n}_{n,n}\left(\genfrac{}{}{0pt}{}{1-\alpha_1,\dots,
      1-\alpha_n}{\gamma_1, \dots, \gamma_{n}};z\right) =
  \int \frac{\diff{}{t}}{2\pi i} z^t \prod_{j=1}^n
  \frac{\Gamma(\alpha_j+t)}{\Gamma(1-\gamma_j+t)} \prod_{h=1}^p \Gamma(\gamma_h-t)\Gamma(1-\gamma_h+t)
\]
for $1 \leq p \leq n$. This integral converges for $|\arg z| < p\,\pi$. 
We introduce the notation
\begin{equation}
  \label{eq:Gp}
  G_p(z) :=  G_p\left(\genfrac{}{}{0pt}{}{\alpha_1,\dots,
      \alpha_n}{\gamma_1, \dots, \gamma_{n}};z\right) := G^{p,n}_{n,n}\left(\genfrac{}{}{0pt}{}{1-\alpha_1,\dots,
      1-\alpha_n}{\gamma_1, \dots,
      \gamma_{n}};(-1)^{p-2}z\right)
\end{equation}
again omitting the parameters when no confusion is possible. These
integrals define functions that are holomorphic (but not necessarily
single--valued) at $z=0$ for $p \geq 1$ and holomorphic,
single--valued at $z=1$ for $p > 1$. Moreover, they are also solutions
to the hypergeometric differential equation (for $\gamma_n=0$). In
particular, if $\gamma_1,\dots,\gamma_n$ are not in resonance, and  if
we reorder the $\gamma_i$'s such that $\gamma_j$ is in the first
position, then 
\[
    G_1\left(\genfrac{}{}{0pt}{}{\alpha_1,\dots,
      \alpha_n}{\gamma_j, \gamma_1, \dots, \gamma_{n}};z\right)  = y_j^*(z) 
\]
with $y_j^*(z)$ as in~\eqref{eq:nonresonantyj} and if we reorder the
$\gamma_i$'s such that $\gamma_i, \gamma_j$ are in the first two
positions we have, 
$$
  G_2\left(\genfrac{}{}{0pt}{}{\alpha_1,\dots,
      \alpha_n}{\gamma_i, \gamma_j,\gamma_1, \dots, \gamma_{n}};z\right) = y_{ij}(z) 
$$
More generally, we have
\begin{lem}
  \label{lem:4}
  \begin{enumerate}
  \item If $\gamma_1,\dots,\gamma_n$ are not in resonance, then
  \[
  G_p(z) = \sum_{j=1}^p \prod_{\substack{k=1\\k\not = j}}^p
  \frac{\pi\e{i\pi (p-1)\gamma_j} }{\sin\pi(\gamma_j-\gamma_k)} y_j^*(z)
  \]
  with $y_j^*(z)$ as in~\eqref{eq:nonresonantyj}.
  \item If $\gamma_1,\dots,\gamma_q$, $2 \leq q \leq n$, are resonant,
    then for $1 \leq p \leq q$:
  \[
    G_p(z) = \e{-2\pi i \gamma_1}\e{i\pi\sum_{j=1}^p\gamma_{j}} (2\pi i)^{p-1} \sum_{j=1}^{p} (-1)^{p-j} \binom{p-1}{p-j} y_j^*(z)  
  \]
  with $y_j^*(z)$ as in~\eqref{eq:resonantyj}.
\end{enumerate}
\end{lem}
\begin{proof}
  \begin{enumerate}
  \item This is essentially Meijer's lemma. We
    present a proof that does not involve evaluating the
    residues. Instead we use the following trigonometric identity. For
    $t \not \in \{\gamma_1,\dots,\gamma_n\}$ we have:
   \begin{equation}
      \label{eq:trigidentity}
  \prod_{j=1}^p \frac{\pi}{\sin\pi(\gamma_j-t)} = \sum_{j=1}^p
  \prod_{\substack{k=1\\k\not = j}}^p
  \frac{\pi}{\sin\pi(\gamma_j-\gamma_k)} \frac{\pi\e{(p-1)\pi
      i(\gamma_j-t)}}{\sin\pi(\gamma_j-t)}
    \end{equation}
  This identity can be proven either by induction on $p$ or by using a
  partial fraction expansion. Then, by consecutively applying Euler's
  identity and~\eqref{eq:trigidentity} to~\eqref{eq:Gp}, we have
  \[
  \begin{aligned}
    G_p(z) 
    &= \int \frac{\diff{}{t}}{2\pi i}\,f_n(t) z^t
    \e{i\pi(p-2)t} \prod_{j=1}^p \frac{\pi}{\sin\pi(\gamma_j-t)}\\
    &= \int \frac{\diff{}{t}}{2\pi i}\,f_n(t) z^t \e{i\pi(p-2)t}
    \sum_{j=1}^p \prod_{\substack{k=1\\k\not=j}}^p
    \frac{\pi}{\sin\pi(\gamma_j-\gamma_k)}
    \frac{\pi\e{i\pi(p-1)(\gamma_j-t)}}{\sin\pi(\gamma_j-t)}\\
    &= \sum_{j=1}^p \prod_{\substack{k=1\\k\not=j}}^p \frac{\pi
      \e{i\pi(p-1)\gamma_j}}{\sin\pi(\gamma_j-\gamma_k)} \int
    \frac{\diff{}{t}}{2\pi i}\,f_n(t) z^t \e{-i\pi t}
    \frac{\pi}{\sin\pi(\gamma_j-t)}\\
  \end{aligned}
  \]
  By~\eqref{eq:nonresonantyj}, the last equation gives the claim.
  \item It is easy to see that
  \[
  \begin{aligned}
    &\sum_{j=1}^{p} (-1)^{p-j} \binom{p-1}{p-j} \int \diff{}{t} \, z^t
    f_n(t) \frac{1}{\left(1-\e{2\pi i(t-\gamma_1)}\right)^j} \\ &= \int
    \diff{}{t} \, z^t f_n(t) \frac{\e{2(p-1)\pi
        i(t-\gamma_1)}}{\left(1-\e{2\pi i(t-\gamma_1)}\right)^p}
  \end{aligned}
  \]
  Using Euler's identity for the Gamma function
  \[
  \Gamma(1+x)\Gamma(-x) = \frac{\pi}{\sin\pi(-x)} = \frac{2\pi i\, \e{i\pi x}}{1-\e{2\pi i x}}
  \]
  the expression on the right hand side can be written as
  \[
  \frac{1}{(2\pi i)^p} \int \diff{}{t} \, z^t f_n(t) \e{(p-2)\pi i(t-\gamma_1)}\left( \Gamma(1-\gamma_1+t)\Gamma(\gamma_1-t) \right)^p
  \]
  Finally, since the $\gamma_1,\dots,\gamma_q$ are resonant, we have for $1 \leq j \leq q$
  \[
    \Gamma(1-\gamma_1+t)\Gamma(t-\gamma_1) = (-1)^{\gamma_{1}-\gamma_{j}}\Gamma(1-\gamma_j+t)\Gamma(\gamma_j-t)
  \]
  so that above expression becomes
  \[
   \frac{\e{2\pi i \gamma_1}\e{-i\pi\sum_{j=1}^p\gamma_{j}}}{(2\pi
     i)^{p-1}} G_p(z) \qedhere
  \]
  \end{enumerate}
\end{proof}

\begin{rem}
  \label{rem:23}
  Lemma~\ref{lem:4}(1) is also valid if there are resonant exponents
  $\gamma_{i_1},\dots,\gamma_{i_p}$. In this case, there is a zero in
  the numerator due to $y_{i_1}^*(z) = \dots = y_{i_p}^*(z)$ with
  $y_{i_j}^*(z)$ given in~\eqref{eq:nonresonantyj}, and a zero of the
  same order in the denominator due to the sine. However, the quotient
  is well--defined as one can see by a de l'H\^opital argument, or
  more easily by the trigonometric
  identity~\eqref{eq:trigidentity}. The left-hand side is always
  well-defined. 
\end{rem}

\begin{prop}
  \label{prop:18}
  If $\gamma_1,\dots,\gamma_q$ are in resonance, then $G_p(z)$, $1\leq p \leq q$, together with the $y_j^*(z)$, $j=q+1,\dots,n$ form a basis of solutions to the hypergeometric differential equation at $z=0$.
\end{prop}
\begin{proof}
  By N\o rlund's result, we know that $\gamma^*_j(z)$, $j=1,\dots,q$
  form a linearly independent set of solutions. The statement then
  follows from Lemma~\ref{lem:4} (b). Indeed, because
  $\binom{p-1}{j-1} = 0$ for $j > p$ and $\binom{p-1}{p-1} =1$, the
  matrix $(A_{pj})$, $1\leq j,p \leq q$ with 
\[
A_{pj} = \e{-2\pi i \gamma_1}\e{i\pi\sum_{j=1}^p\gamma_{j}} (2\pi i)^{p-1}
(-1)^{p-j} \binom{p-1}{p-j} 
\]
 is upper triangular, with non vanishing entries on the diagonal, hence invertible. 
\end{proof}

\begin{cor}
  \label{cor:19}
  \begin{enumerate}
  \item If $\gamma_1,\dots,\gamma_n$ are not in resonance, then
    $y_{ij}(z)$, for fixed $i$, $j=1,\dots,i-1,i+1,\dots,n$ and
    $\xi_n(z)$ form a basis of solutions at $z=1$.  
  \item If
    $\gamma_1,\dots,\gamma_q$ are resonant, then $\xi_n(z), G_p(z)$,
    $p=2,\dots, q$, and $y_{1j}(z)$, $j=q+1,\dots, n$ form a basis of
    solutions at $z=1$.
  \end{enumerate}
\end{cor}
Therefore, we will refer to the solutions $G_p(z), p>1$, in the
resonant case as logarithmic solutions, as well.

Lemma~\ref{lem:4}, Proposition~\ref{prop:18} and
Corollary~\ref{cor:19} almost completely yield the analytic
continuation from $z=0$ to $z=1$. We only lack the relation of
$\xi_n(z)$ to the solutions $y_j^*(z)$ at $z=0$. This will be reviewed
in Section~\ref{sec:solution-xi_n-1}. However, in order to compare to
the power series solutions obtained from the Frobenius method, we
need to determine the series expansions of $y_{ij}(z)$ in the
nonresonant case, and $G_p(z)$ in the resonant case at $z=1$. This
will be discussed in Section~\ref{sec:series-expansions}.

For completeness, we finish this section by reviewing the analytic
continuation from $z_0=0$ to $z_0=\infty$. The solutions $y_j^*(z)$ at
$z_0=0$ and $\overline{y}_k^*(z)$ at $z_0=\infty$ are related by
$y_j^*(z) = \sum_{k=1}^n M_{0\infty,jk} \overline{y}_k^*(z)$. The
entries of the matrix $M_{0,\infty}$ can be determined by converting
the Gamma functions in the integral representations of $y_j^*(z)$ and
$\overline{y}_k^*(z)$ into sine functions using the Euler identity and
then applying a partial fraction expansion. N\o rlund gives explicit
formulas for the coefficients $M_{0\infty,jk}$ in both the resonant
and nonresonant cases~\cite{Norlund:1955ab}. 


\section{Analytic continuation for $n=2$}
\label{sec:class-hyperg-funct}

In this section we review the well-known results for the analytic
continuation of the classical hypergeometric function. We present them
on one hand for completeness, and on the other hand because we will
need them in the higher order situation. The case $n=2$ is special
because it is the only case where the differential equation at $z=1$ is
also in a hypergeometric form. Indeed, it is easy to see that under
the change of variables $z \to 1-z$ in~\eqref{eq:DE}, the indices
$\gamma_1$ and $\beta_2$ are interchanged.

\begin{lem}
  \label{lem:5}
  \begin{enumerate}
  \item For $\lambda_1,\lambda_2,\mu_1,\mu_2 \in \mC$ such that $\Re
    (\lambda_j+\mu_j) \not \in \mZ_{\leq 0}$ for $i,j=1,2$, 
  \[
  \begin{aligned}
    &\int \frac{\diff{}{t}}{2\pi
      i}\Gamma(\lambda_1+t) \Gamma(\lambda_2+t) \Gamma(\mu_1-t)\Gamma(\mu_2-t) \\
    & =
    \frac{\Gamma(\lambda_1+\mu_1)\Gamma(\lambda_1+\mu_2)\Gamma(\lambda_2+\mu_1)\Gamma(\lambda_2+\mu_2)}{\Gamma(\lambda_1+\lambda_2+\mu_1+\mu_2)}
  \end{aligned}
  \]
  \item For $\lambda_1,\lambda_2,\mu,\nu \in \mC$ such that
    $\Re(\nu-\lambda_1-\lambda_2-\mu) > 0$, 
  \[
  \begin{aligned}
    &\int \frac{\diff{}{s}}{2 \pi i}\e{\pm \pi i
      s}\frac{\Gamma(\lambda_1+s)\Gamma(\lambda_2+s) \Gamma(\mu-s)}{\Gamma(\nu+s)} \\
    & = \e{\pi i \mu}
    \frac{\Gamma(\lambda_1+\mu)\Gamma(\lambda_2+\mu)\Gamma(\nu-\lambda_1-\lambda_2-\mu)}{\Gamma(\nu-\lambda_1)\Gamma(\nu-\lambda_2)}.
  \end{aligned}
  \]
\end{enumerate}
\end{lem}
\begin{proof}
  The first identity is known as Barnes' first lemma. For a nice proof
  of both identities only involving integral representations
  see~\cite{Jantzen:2012cb}. 
\end{proof}

\begin{rem}
  \label{rem:32}
  As a consequence of Lemma~\ref{lem:5}(2) we obtain the well--known
  identity of Gauss:
  \begin{equation}
   \label{eq:Gauss}
    {}_{2}F_{1}\left(\genfrac{}{}{0pt}{0}{a,b}{c};1\right) =
    \frac{\Gamma(c)\Gamma(c-a-b)}{\Gamma(c-a)\Gamma(c-b)} ,
  \end{equation}
  valid for $\Re (c-a-b) > 0$. 
\end{rem}

\begin{prop}
 \label{prop:8} 
   \[
    \begin{aligned}
    &\int \frac{\diff{}{s} }{2\pi i} \,z^s\e{-i\pi s} \frac{\Gamma(a+s)
      \Gamma(b+s) \Gamma(-s)}{\Gamma(c+s)} \\
    &=
    \frac{1}{\Gamma(c-a)\Gamma(c-b)} \int \frac{\diff{}{t}}{2\pi i}
    \Gamma(a+t) \Gamma(b+t) \Gamma(c-a-b-t) \Gamma(-t)(1-z)^t
  \end{aligned}
   \]
   or in other words
   \[
  \frac{\Gamma(a)\Gamma(b)}{\Gamma(c)} {}_2F_1 \left(\genfrac{}{}{0pt}{0}{a,b}{c};z\right) = \frac{1}{\Gamma(c-a)\Gamma(c-b)} G_2 \left(\genfrac{}{}{0pt}{0}{a,b}{c-a-b,0};1-z\right)
   \]
\end{prop}
\begin{proof}
  This is a well--known result. 
  For the sake of exposition we give a proof. Setting $\lambda_1=a, \lambda_2=b,
  \mu_1=s, \mu_2=c-a-b$, Barnes' first lemma~\ref{lem:5}(1) can be
  rewritten as
  \[
   \frac{\Gamma(a+s)\Gamma(b+s)}{\Gamma(c+s)} =
   \frac{1}{\Gamma(c-a)\Gamma(c-b)} \int_{k-i \infty}^{k+i
     \infty} \frac{\diff{}{t}}{2\pi i} \Gamma(a+t) \Gamma(b+t)
   \Gamma(s-t) \Gamma(c-a-b-t) 
  \]
  where we have chosen the integration contour as the line $\Re t =
  k$ for $0 < k < 1$, possibly indented such that poles of $\Gamma(a+t)$ and
  $\Gamma(b+t)$ lie to the right, and the poles of $\Gamma(s-t)$ and
  $\Gamma(c-a-b-t)$ lie to the left of this line.
  The left hand side of the claim then becomes
  \[
  \begin{aligned}
    &\int \frac{\diff{}{s}}{2\pi i}\, z^s
    \e{-i\pi s} \frac{\Gamma(-s)}{\Gamma(c-a)\Gamma(c-b)}\\
   & 
  \cdot \int_{k-i \infty}^{k+i \infty} \frac{\diff{}{t}}{2\pi i}\,\Gamma(a+t)
  \Gamma(b+t) \Gamma(s-t)\Gamma(c-a-b-t) \\
  \end{aligned}
  \]
  If $k$ is chosen such that the lower bound of the distance between the
$s$ contour and the $t$ contour is positive, i.e. nonzero, then the
order of integration may be interchanged. Then the double integral
takes the form
  $$
  \begin{aligned}
     &\frac{1}{\Gamma(c-a)\Gamma(c-b)}\int_{k-i \infty}^{k+i \infty} \frac{\diff{}{t}}{2\pi i}\,\Gamma(a+t)
  \Gamma(b+t) \Gamma(c-a-b-t) \\ 
   &\cdot \int \frac{\diff{}{s}}{2\pi i}\, z^s
    \e{-i\pi s} \Gamma(-s)\Gamma(s-t)\\
  \end{aligned}
  $$
  Now the $s$ integral can be evaluated using the identity~\cite{Erdelyi:1953ab,Slater:1966ab}
  \begin{equation}
   \label{eq:binomialidentity}
   \int \frac{\diff{}{s}}{2\pi i}\, z^s
    \e{-i\pi s} \Gamma(-s)\Gamma(s-t) = \Gamma(-t) (1-z)^t
    \qedhere
  \end{equation}
\end{proof}

From this proposition, the analytic continuation can be performed for
all possible values of $a,b,c$, in particular for $c\in\mZ$, and/or
$c-a-b \in \mZ$ which correspond to the resonant situation. We refer
to~\cite{Norlund:1963ab} for details, where the function
$\frac{(-1)^{c-1}\Gamma(c)}{\Gamma(a)\Gamma(b)} G_2
\left(\genfrac{}{}{0pt}{0}{a,b}{1-c,0};z\right)$ was denoted
$g(a,b,c;z)$.

\section{Recurrence relations by order}
\label{sec:recurr-relat-order}

The general strategy for the analytic continuation of ...  is the
observation that each of the solutions $y_i^*(z), y_{ij}(z), G_p(z),
\xi_n(z)$ given in Section~\ref{sec:hyperg-funct} can be expressed in
terms of the solutions of the same type of hypergeometric equations of
order one less. This recurrence allows us to reduce the problem of the analytic
continuation to the one of order two which is well-known, see
Section~\ref{sec:class-hyperg-funct}. The reduction step will be discussed in
Section~\ref{sec:analyt-cont-z=1}.

\subsection{N\o rlund's recurrence for the solution $\xi_n$ }
\label{sec:no-rlunds-recurrence}

We begin with the special solution at $z=1$ associated to the exponent
$\beta_n$. For $\beta_n \not \in \mZ_{<0}$ N\o rlund has proven the
following recurrence as well as several
consequences~\cite{Norlund:1955ab}. These consequences will be used in
Section~\ref{sec:logar-solut}. 
\begin{lem}
 \label{lem:9} 
  If $\Re \beta_n > \Re \beta_{n-1} > -1$, then
  \[
  \xi_n(z) = \frac{\Gamma(\beta_n+1)}{\Gamma(1-\alpha_n-\gamma_n)\Gamma(\beta_{n-1}+1)} z^{\gamma_n} \int_{z}^1 t^{\alpha_n-1}(t-z)^{-\alpha_n-\gamma_n}\xi_{n-1}(t)\diff{}{t}
  \]
\end{lem}
\begin{proof}
  A proof is given in~\cite{Norlund:1955ab}. It relies on showing that
  the generalized Euler integral on the right hand side is a solution
  to~\eqref{eq:DE}, and that its series expansion around $z=1$ corresponds to the
  solution with index $\beta_n$.  
\end{proof}

\begin{lem}
  \label{lem:3}
  If $\Re \beta_n > -1,
  \Re(x+\gamma_s) > 0, s=1,\dots,n$, then
  \[
  \int_0^1 t^{x-1} \xi_n(t) \diff{}{t} =
  \frac{\Gamma(\beta_n+1)\Gamma(x+\gamma_n)}{\Gamma(\beta_{n-1}+1)\Gamma(x-\alpha_n+1)}
    \int_0^1 u^{x-1}\xi_{n-1}(u) \diff{}{u}
  \]
\end{lem}
\begin{proof}
  This is again due to~\cite{Norlund:1955ab}. The integral on the left
  hand side converges for $\Re \beta_n > -1, \Re(x+\gamma_s) > 0,
  s=1,\dots,n$, since it is a linear function of the solutions
  $y_1^*(z), \dots, y_n^*(z)$. Substituting Lemma~\ref{lem:9} yields
  \[
  \begin{aligned}
    &\int_0^1 t^{x-1} \xi_n(t) \diff{}{t} \\
    &=
    \frac{\Gamma(\beta_n+1)}{\Gamma(1-\alpha_n-\gamma_n)\Gamma(\beta_{n-1}+1)}
    \int_0^1 t^{x-1+\gamma_n} \int_{t}^1
    u^{\alpha_n-1}(u-t)^{-\alpha_n-\gamma_n}\xi_{n-1}(u)\diff{}{u}
    \diff{}{t}\\
    &= \frac{\Gamma(\beta_n+1)}{\Gamma(1-\alpha_n-\gamma_n)\Gamma(\beta_{n-1}+1)}
    \int_{0}^1
    u^{\alpha_n-1}\xi_{n-1}(u)\int_0^u t^{x-1+\gamma_n} (u-t)^{-\alpha_n-\gamma_n}
    \diff{}{t}\diff{}{u}\\
    &= \frac{\Gamma(\beta_n+1)}{\Gamma(1-\alpha_n-\gamma_n)\Gamma(\beta_{n-1}+1)}
    \int_{0}^1
    u^{x-1}\xi_{n-1}(u) \diff{}{u} \int_0^1 t^{x-1+\gamma_n} (1-t)^{-\alpha_n-\gamma_n}
    \diff{}{t}\\
    &=
    \frac{\Gamma(\beta_n+1)}{\Gamma(1-\alpha_n-\gamma_n)\Gamma(\beta_{n-1}+1)}
    \frac{\Gamma(x+\gamma_n)\Gamma(-\alpha_n-\gamma_n+1)}{\Gamma(x-\alpha_n+1)} 
    \int_{0}^1 u^{x-1}\xi_{n-1}(u) \diff{}{u} \\
    &=
    \frac{\Gamma(\beta_n+1) \Gamma(x+\gamma_n)}{\Gamma(\beta_{n-1}+1)
      \Gamma(x-\alpha_n+1)} \int_{0}^1 u^{x-1}\xi_{n-1}(u) \diff{}{u} \\
  \end{aligned}
  \]
For the second equality, we have used a generalization of a formula of
Dichichlet's for changing the order of
integration~\cite{Hurwitz:1908ab}. Moreover, ~\cite{Norlund:1955ab}
shows that  the condition $\Re\beta_n > -1$ can be relaxed to $\beta_n \not \in
\mZ_{<0}$. 
\end{proof}

\begin{lem}
  \label{lem:1}
  For $1\leq p < n$, and if $\Re \beta_n > \Re\beta_p,
  \Re(x+\alpha_s) > 0, s=p+1, \dots,n$, then
  \[
  \int_0^z t^{x-1} \bar\xi_{n-p}\left(\genfrac{}{}{0pt}{0}{\alpha_{p+1},\dots,\alpha_n}{\gamma_{p+1},\dots,\gamma_n};\frac{z}{t}\right) \diff{}{t} = z^x\Gamma(\beta_n-\beta_p)
    \prod_{s=p+1}^n \frac{\Gamma(\alpha_s+x)}{\Gamma(1-\gamma_s+x)}   
  \]
\end{lem}
\begin{proof}
  This is also due to~\cite{Norlund:1955ab}. Applying Lemma~\ref{lem:3} $n-1$ times yields
  \[
  \int_0^1 t^{x-1} \xi_n(t) \diff{}{t} = \Gamma(\beta_n+1) \prod_{s=1}^n \frac{\Gamma(x+\gamma_s)}{\Gamma(x-\alpha_s+1)}
  \]
  for $\Re(\beta_n) > -1$ and $\Re(x+\gamma_s) > 0$, $s=1,\dots,n$. 
  Recall from Section~\ref{sec:hyperg-diff-equat} that interchanging
  $\alpha_j$ with $\gamma_j$, $j=1,\dots,n$ and replacing $t$ with
  $1/t$ turns $\xi_n(z)$ 
  into $\bar \xi_n(z)$. Applying these operations together with the
  replacement of $x$ by $-x$ to the above integral yields
\[
  \int_1^\infty t^{x-1} \bar\xi_n(t) \diff{}{t} = \Gamma(\beta_n+1)
    \prod_{s=1}^n \frac{\Gamma(\alpha_s-x)}{\Gamma(1-\gamma_s-x)}  
\]
for $\Re \beta_n > -1, \Re(\alpha_s-x) > 0, s=1, \dots,n$. Replacing
$t$ by $\frac{z}{t}$ and $x$ by $-x$ and exhibiting the dependence
of $\bar\xi_n(z)$ on the parameters $\alpha$ and $\gamma$ yields
\[
  \int_0^z t^{x-1} \bar\xi_n\left(\genfrac{}{}{0pt}{0}{\alpha_1,\dots,\alpha_n}{\gamma_1,\dots,\gamma_n};\frac{z}{t}\right) \diff{}{t} = z^x\Gamma(\beta_n+1)
    \prod_{s=1}^n \frac{\Gamma(\alpha_s+x)}{\Gamma(1-\gamma_s+x)}    
\]
Replacing $n$ by $n-p$ for $p<n$ and restricting to the last $n-p$
parameters $\alpha_{p+1},\dots,\alpha_n$,
$\gamma_{p+1},\dots,\gamma_n$, $\beta_{n-p}$ can be written as
$\beta_n-\beta_p-1$, and the claim follows.
\end{proof}

\begin{lem}
  \label{lem:2}
  For $1\leq p < n$, $v \in \mC \setminus \mZ_{\leq 0}$, and if $\Re
  \beta_n > \Re\beta_p$, $\Re(x+\alpha_s) > 0, s=p+1, \dots,n$, then
  \[
    \begin{aligned}
    &\int_0^z \diff{}{t}\, t^{\gamma_1-1}(1-t)^v
    \bar\xi_{n-p}\left(\genfrac{}{}{0pt}{0}{\alpha_{p+1},\dots,\alpha_n}{\gamma_{p+1},\dots,\gamma_n};\frac{z}{t}\right)
    \\
    &= \Gamma(\beta_n-\beta_p) z^{\gamma_1} \int
    \frac{\diff{}{u}}{2\pi i} \e{-i\pi u} z^u
    \frac{\Gamma(-v+u)\Gamma(-u)}{\Gamma(-v)} \prod_{s=p+1}^n
    \frac{\Gamma(\alpha_s+\gamma_1+u)}{\Gamma(1-\gamma_s+\gamma_1+u)}
  \end{aligned}
  \]
\end{lem}
\begin{proof}
  We apply the binomial identity~\eqref{eq:binomialidentity} to
  $(1-t)^{v}$ in the integrand on the left hand side 
  and obtain
\[
  \int \frac{\diff{}{u}}{2\pi i}\, 
    \e{-i\pi u} \frac{\Gamma(-u)\Gamma(u-v)}{\Gamma(-v)} \int_0^z \diff{}{t}\, t^{\gamma_1+u-1} 
    \bar\xi_{n-p}\left(\genfrac{}{}{0pt}{0}{\alpha_{p+1},\dots,\alpha_n}{\gamma_{p+1},\dots,\gamma_n};\frac{z}{t}\right)
\]
The claim follows from Lemma 4.3 and analytic continuation.
\end{proof}

In order to state the following two propositions, we choose $p$ of the
parameters $\alpha_1,\dots,\alpha_n$ and $p$ of the parameters
$\gamma_1,\dots,\gamma_n$, say for simplicity
$\alpha_1,\dots,\alpha_p$ and $\gamma_1,\dots,\gamma_p$. Our results
do not depend on this choice, however, the complexity of explicit
calculations can depend on it. Then we consider the
hypergeometric differential equation of order $p$ with exponents
$\alpha_1,\dots,\alpha_p$ and $\gamma_1,\dots,\gamma_p$. For its
solutions we introduce the notation 
\[
  \widetilde y_j^*(z) := \prod_{s=1}^p
    \frac{\Gamma(\alpha_s+\gamma_j)}{\Gamma(1-\gamma_s+\gamma_j)} z^{\gamma_j}
    {}_pF_{p-1}\left(\genfrac{}{}{0pt}{0}{\alpha_{1}+\gamma_j,\dots,\alpha_p+\gamma_j}{1-\gamma_{1}+\gamma_j,\widehat{\dots},1-\gamma_p+\gamma_j};z\right)
\]
for $1 \leq j \leq p$, as well as the special solution
\begin{equation}
  \widetilde G_p(z) := \widetilde
  G_p\left(\genfrac{}{}{0pt}{}{\alpha_1,\dots, \alpha_p}{\gamma_1,
      \dots, \gamma_{p}};z\right) :=
  G^{p,p}_{p,p}\left(\genfrac{}{}{0pt}{}{1-\alpha_1,\dots,
      1-\alpha_p}{\gamma_1, \dots,
      \gamma_{p}};(-1)^{p-2}z\right)
  \label{eq:Gtilde}
\end{equation}

\begin{prop}
  \label{prop:7}
  For $1 \leq p \leq n-1$, $1 \leq j \leq p$, and if $\Re\beta_n >
  \Re\beta_p$, $\Re(\gamma_j + \alpha_s) > 0$, $s=p+1,\dots,n$, we have
  \[
    y_j^*(z) 
   = \frac{1}{\Gamma(\beta_n-\beta_p)}
    \int_0^z \frac{\diff{}{t}}{t}\, \widetilde y_j^*(t)
    \bar\xi_{n-p}\left(\genfrac{}{}{0pt}{0}{\alpha_{p+1},\dots,\alpha_n}{\gamma_{p+1},\dots,\gamma_n};\frac{z}{t}\right)\\
  \]
\end{prop}
\begin{proof}
  We use the Mellin--Barnes integral representation for $\widetilde
  y_j^*(t)$ to write
  \[
  \begin{aligned}
    &\int_0^z \frac{\diff{}{t}}{t}\, \widetilde y_j^*(t)
    \bar\xi_{n-p}\left(\genfrac{}{}{0pt}{0}{\alpha_{p+1},\dots,\alpha_n}{\gamma_{p+1},\dots,\gamma_n};\frac{z}{t}\right)\\
    &= \int_0^z \frac{\diff{}{t}}{t}\, \int \frac{\diff{}{u}}{2\pi i}\, \e{-i\pi u} t^u f_p(u)
    \Gamma(\gamma_j-u)\Gamma(1-\gamma_j+u)
    \bar\xi_{n-p}\left(\genfrac{}{}{0pt}{0}{\alpha_{p+1},\dots,\alpha_n}{\gamma_{p+1},\dots,\gamma_n};\frac{z}{t}\right)\\
  \end{aligned}
  \]
  Due to the conditions $\Re(\alpha_s+\gamma_j) > 0, s=p+1,\dots,n$,
  we can interchange the integrals and apply Lemma~\ref{lem:1}. Then
  this expression becomes 
  \[
  \begin{aligned}
    &\int \frac{\diff{}{u}}{2\pi i}\, \e{-i\pi u} f_p(u)
    \Gamma(\gamma_j-u)\Gamma(1-\gamma_j+u) \int_0^z
    \frac{\diff{}{t}}{t}\, t^u
    \bar\xi_{n-p}\left(\genfrac{}{}{0pt}{0}{\alpha_{p+1},\dots,\alpha_n}{\gamma_{p+1},\dots,\gamma_n};\frac{z}{t}\right)\\
    &= \int \frac{\diff{}{u}}{2\pi i}\, \e{-i\pi u} f_p(u)
    \Gamma(\gamma_j-u)\Gamma(1-\gamma_j+u) z^u\Gamma(\beta_n-\beta_p)
    \prod_{s=p+1}^n \frac{\Gamma(\alpha_s+u)}{\Gamma(1-\gamma_s+u)}  \\
    &= \Gamma(\beta_n-\beta_p) \int \frac{\diff{}{u}}{2\pi i}\,
    \e{-i\pi u} z^u f_n(u)
    \Gamma(\gamma_j-u)\Gamma(1-\gamma_j+u) \\
  \end{aligned}
  \]
  The integral in the last expression is $y_j^*(z)$.
\end{proof}
Proposition~\ref{prop:7} has the following analog for $G_p(z)$. 
\begin{prop}
 \label{prop:6} 
 For any $1 \leq p < n$, and if $\Re \beta_n > \Re\beta_p$,
 $\Re(\alpha_s+\gamma_j)>0$, $j=1\dots,p$, $s=p+1,\dots,n$ then  
  \[
    G_p \left(\genfrac{}{}{0pt}{0}{\alpha_{1},\dots,\alpha_n}{\gamma_{1},\dots,\gamma_n};z\right) = \frac{1}{\Gamma(\beta_n-\beta_p)}
    \int_0^z \frac{\diff{}{t}}{t}\, \widetilde G_p \left(\genfrac{}{}{0pt}{0}{\alpha_{1},\dots,\alpha_p}{\gamma_{1},\dots,\gamma_p};t\right)\bar\xi_{n-p}\left(\genfrac{}{}{0pt}{0}{\alpha_{p+1},\dots,\alpha_n}{\gamma_{p+1},\dots,\gamma_n};\frac{z}{t}\right)\\
  \]
\end{prop}
\begin{proof}
  By Lemma~\ref{lem:4}(1) and Remark~\ref{rem:23} we have $G_p(z) =
  \sum_{j=1}^p A_{pj} y_j^*(z)$ with $A_{pj} = \prod_{\substack{k=1\\k\not=j}}^p
    \frac{\pi\e{i\pi(p-1)\gamma_j}}{\sin\pi(\gamma_j-\gamma_k)}$. 
  We apply Proposition~\ref{prop:7} and interchange the finite sum and the integral to obtain
  \[
  \begin{aligned}
    G_p(z) 
    &= \frac{1}{\Gamma(\beta_n-\beta_p)}
    \int_0^z \frac{\diff{}{t}}{t}\, \sum_{j=1}^p A_{pj} \widetilde y_j^*(t)
    \bar\xi_{n-p}\left(\genfrac{}{}{0pt}{0}{\alpha_{p+1},\dots,\alpha_n}{\gamma_{p+1},\dots,\gamma_n};\frac{z}{t}\right)\\ 
  \end{aligned}
  \]
  From Lemma~\ref{lem:4}(1) for $n=p$ and Remark~\ref{rem:23} we also
  have $\widetilde G_p(z) = \sum_{j=1}^p A_{pj} \widetilde
  y_j^*(z)$. The claim follows. 
\end{proof}
For $p=2$ and $\gamma_1 - \gamma_2 \not \in \mZ$ these two
propositions are due to N\o rlund~\cite{Norlund:1955ab}.

\subsection{B\"uhring's recurrence for ${}_nF_{n-1}$}
\label{sec:buhr-recurr-_nf_n}

In this Section we review B\"uhring's recurrence for ${}_nF_{n-1}(z)$
from the point of view of integral
representations~\cite{Buehring:1992ab}. It is then applicable to each
of the $y_j^*(z)$ in the nonresonant case~\eqref{eq:nonresonantyj}, as
well as to $G_1(z)$ in the resonant case. Since this requires the 
parameters of ${}_nF_{n-1}(z)$ to take various sets of different values, we use
$a_1,\dots,a_n$ and $b_1,\dots,b_{n-1}$, as well as
\begin{equation}
  \label{eq:c}
  c = \sum_{j=1}^{n-1} b_j - \sum_{j=1}^n a_j.
\end{equation}
instead of $\alpha_1,\dots,\alpha_n$, 
$1-\gamma_1,\dots,1-\gamma_{n-1}$, and $\beta_n$, respectively.

\begin{lem}
  \label{lem:14}
\[
  \begin{aligned}  
& \,{}_n F_{n-1}\left(\genfrac{}{}{0pt}{}{a_1,\dots, a_n}{b_1, \dots, b_{n-1}};z\right) \\
& = \frac{\Gamma(b_{n-2})\Gamma(b_{n-1})}{\Gamma(a_n)\Gamma(b_{n-1}-a_{n})\Gamma(b_{n-2}-a_{n})}\\
& \phantom{=} \cdot \int \frac{\diff{}{t}}{2\pi i} \,\e{\pm\pi i t}\frac{\Gamma(-t)\Gamma(b_{n-1}-a_{n}+t)\Gamma(b_{n-2}-a_{n}+t)}{\Gamma(b_{n-1}+b_{n-2} - a_n + t)}
 \\
& \phantom{=} \cdot \,{}_{n-1} F_{n-2}\left(\genfrac{}{}{0pt}{}{a_1,\dots, a_{n-1}}{b_1, \dots, b_{n-3}, b_{n-1}+b_{n-2} - a_n + t};z\right)
\end{aligned}
\]
\end{lem}
\begin{proof}
  Consider the Mellin Barnes integral representation of
  ${}_{n}F_{n-1}(z)$ given in~\eqref{eq:nonresonantyj}
  \begin{equation*}
  \frac{\Gamma(a_1)\dots\Gamma(a_n)}{\Gamma(b_1)\dots\Gamma(b_{n-1})}
\,{}_n F_{n-1}\left(\genfrac{}{}{0pt}{}{a_1,\dots, a_n}{b_1, \dots,
    b_{n-1}};z\right) = \int \frac{\diff{}{s}}{2\pi i} \,z^s \e{-i\pi
  s} \frac{\Gamma(a_1+s)\dotsm\Gamma(a_n+s)\Gamma(-s)}{\Gamma(b_1+s)\dotsm\Gamma(b_{n-1}+s)}
  \label{eq:MellinBarnes}
  \end{equation*}
  We apply Lemma~\ref{lem:5}(2) with $\lambda_1=b_{n-1}-a_n,\lambda_2 = b_{n-2} -
  a_n, \mu=0, \nu=b_{n-1}+b_{n-2}-a_n+t, s=t$ to obtain
\begin{equation*}
\begin{aligned}
& \frac{\Gamma(a_1)\dotsm\Gamma(a_n)}{\Gamma(b_1)\dotsm\Gamma(b_{n-1})} \,{}_n F_{n-1}\left(\genfrac{}{}{0pt}{}{a_1,\dots, a_n}{b_1, \dots, b_{n-1}};z\right) \\
& = \frac{1}{\Gamma(b_{n-1}-a_{n})\Gamma(b_{n-2}-a_{n})} \int
\frac{\diff{}{s}}{2\pi i}\, z^s \e{-i\pi s} \frac{\Gamma(a_1+s)\dotsm\Gamma(a_{n-1}+s)\Gamma(-s)}{\Gamma(b_1+s)\dotsm\Gamma(b_{n-3}+s)}\\
&\phantom{=} \cdot\int \frac{\diff{}{t}}{2\pi i}\,\e{\pm\pi i t} \frac{\Gamma(b_{n-1}-a_{n}+t)\Gamma(b_{n-2}-a_{n}+t) \Gamma(-t)}{\Gamma(b_{n-1}+b_{n-2} - a_n + s + t)}
\end{aligned}
\end{equation*}
which is valid for $\Re (a_n+t) > 0$. This justifies interchanging
the order of integration as in the proof of Proposition~\ref{prop:8}
and yields, after extension to the whole $t$ plane by analytic continuation of the
parameters, the desired result.
\[
  \begin{aligned}
& \frac{\Gamma(a_1)\dotsm\Gamma(a_n)}{\Gamma(b_1)\dotsm\Gamma(b_{n-1})} \,{}_n F_{n-1}\left(\genfrac{}{}{0pt}{}{a_1,\dots, a_n}{b_1, \dots, b_{n-1}};z\right) \\
& = \frac{1}{\Gamma(b_{n-1}-a_{n})\Gamma(b_{n-2}-a_{n})} \int
\frac{\diff{}{t}}{2\pi i}\, \e{\pm\pi i t}\Gamma(b_{n-1}-a_{n}+t)\Gamma(b_{n-2}-a_{n}+t) \Gamma(-t)\\
&\phantom{=} \int \frac{\diff{}{s}}{2\pi i}\,z^s \e{-i\pi s}\frac{\Gamma(a_1+s)\dotsm\Gamma(a_{n-1}+s)\Gamma(-s)}{\Gamma(b_1+s)\dotsm\Gamma(b_{n-3}+s)\Gamma(b_{n-1}+b_{n-2} - a_n + s + t)}\\
& = \frac{\Gamma(a_1)\dotsm\Gamma(a_{n-1})}{\Gamma(b_1)\dotsm\Gamma(b_{n-3})\Gamma(b_{n-1}-a_{n})\Gamma(b_{n-2}-a_{n})}\\
& \phantom{=} \cdot \int \frac{\diff{}{t}}{2\pi i} \,\e{\pm\pi i t} \frac{\Gamma(-t)\Gamma(b_{n-1}-a_{n}+t)\Gamma(b_{n-2}-a_{n}+t)}{\Gamma(b_{n-1}+b_{n-2} - a_n + t)}\\
& \phantom{=} \cdot 
 \,{}_{n-1} F_{n-2}\left(\genfrac{}{}{0pt}{}{a_1,\dots, a_{n-1}}{b_1,
     \dots, b_{n-3}, b_{n-1}+b_{n-2} - a_n + t};z\right) 
  \end{aligned}\qedhere
\]
\end{proof}

Using this lemma, we find the following integral representation of
${}_nF_{n-1}(z)$ in terms of ${}_2F_1(z)$: 
\begin{prop}
  \label{prop:13}
  \[
  \begin{aligned}
     & \,{}_n F_{n-1}\left(\genfrac{}{}{0pt}{}{a_1,\dots, a_n}{b_1, \dots, b_{n-1}};z\right) \\
    & =
    \frac{\Gamma(b_{n-1}) \cdots \Gamma(b_{1})}{\Gamma(a_n)
      \cdots \Gamma(a_{1})} \int \frac{\diff{}{u}}{2\pi i} \widetilde{A}^{(n)}(u) \frac{\Gamma(a_1)\Gamma(a_2)}{\Gamma(c+a_1+a_2+u)} \,{}_{2}
    F_{1}\left(\genfrac{}{}{0pt}{}{a_1, a_{2}}{c+a_1+a_2+u};z\right)
  \end{aligned}
  \]
  with
  \begin{align*}
    &\widetilde{A}^{(n)}(u)\\ &=
    \frac{1}{\Gamma(b_{n-1}-a_{n})\Gamma(b_{n-2}-a_{n})
      \Gamma(b_{n-3}-a_{n-1})\cdots \Gamma(b_1-a_3)}\\
    &\phantom{=} \cdot\int \frac{\diff{}{u_{n-2}}}{2\pi i} \e{\pm\pi i
      u_{n-2}}\frac{\Gamma(-u_{n-2})\Gamma(b_{n-1}-a_{n}+u_{n-2})\Gamma(b_{n-2}-a_{n}+u_{n-2})}{\Gamma(b_{n-1}+b_{n-2}
      - a_n -a_{n-1} + u_{n-2})} \\
    & \phantom{=} \cdot \int \frac{\diff{}{u_{n-3}}}{2\pi i} \e{\pm\pi
      i
      (u_{n-3}-u_{n-2})}\frac{\Gamma(u_{n-2}-u_{n-3})\Gamma(b_{n-1}+b_{n-2}
      - a_n -a_{n-1}+
      u_{n-3})}{\Gamma(b_{n-1}+b_{n-2}
      +b_{n-3}- a_n - a_{n-1} - a_{n-2} + u_{n-3})} \\
    & \phantom{=\int \frac{\diff{}{u_{n-3}}}{2\pi i}}
    \cdot \Gamma(b_{n-3}-a_{n-1}+u_{n-3}-u_{n-2})\\
    & \phantom{=} \cdot \quad \cdots\\
    & \phantom{=} \cdot \int \frac{\diff{}{u_{2}}}{2\pi i} \e{\pm\pi i
      (u_{2}-u_{3})}\frac{\Gamma(u_{3}-u_{2})\Gamma(b_{n-1}+\dots +
      b_{3} - a_n -\dots - a_{4}+
      u_{2})}{\Gamma(b_{n-1}+\dots
      +b_{2}- a_n -\dots - a_{3} + u_{2})} \\
    & \phantom{=\int \frac{\diff{}{u_{2}}}{2\pi i}}
    \cdot \Gamma(b_{2}-a_{4}+u_{2}-u_{3})\\
    & \phantom{=} \cdot \e{\pm\pi i
      (u-u_2)}\Gamma(u_2-u)\Gamma(c+a_1+a_2-b_1+u)\Gamma(b_{1}-a_{3}+u-u_2)
  \end{align*}
\end{prop}
For the applications in Section~\ref{sec:applications} we give the first few cases explicitly:
  \begin{align}
    \label{eq:A3t}
    \widetilde{A}^{(3)}(u) &= \e{\pm\pi i
     u}\frac{\Gamma(c+a_1+a_2-b_1+u)\Gamma(b_{1}-a_{3}+u)
    \Gamma(-u)}{\Gamma(b_{2}-a_{3})\Gamma(b_{1}-a_{3})}\\ 
   \label{eq:A4t}
    \widetilde{A}^{(4)}(u) &=
    \frac{1}{\Gamma(b_{3}-a_{4})\Gamma(b_{2}-a_{4})
      \Gamma(b_1-a_3)}\\
    &\phantom{=} \cdot \int \frac{\diff{}{u_{2}}}{2\pi i} \e{\pm\pi i
      u_{2}}\frac{\Gamma(-u_{2})\Gamma(b_{3} - a_4 +
      u_{2})\Gamma(b_{2}-a_{4}+u_{2})}{\Gamma(b_{3}+b_{2}
      - a_4 - a_{3} + u_{2})} \notag\\
    & \phantom{=} \cdot \e{\pm\pi i
      (u-u_2)}\Gamma(u_2-u)\Gamma(c+a_1+a_2-b_1+u)\Gamma(b_{1}-a_{3}+u-u_2) \notag
  \end{align}
  \begin{proof}
    Repeated application of Lemma~\ref{lem:14} yields
  \begin{align*}
    & \,{}_n F_{n-1}\left(\genfrac{}{}{0pt}{}{a_1,\dots, a_n}{b_1, \dots, b_{n-1}};z\right) \\
    & =
    \frac{\Gamma(b_{n-1}) \Gamma(b_{n-2}) \Gamma(b_{n-3})}{\Gamma(a_n) \Gamma(a_{n-1})\Gamma(b_{n-1}-a_{n})\Gamma(b_{n-2}-a_{n}) \Gamma(b_{n-3}-a_{n-1})}\\
    &\phantom{=} \cdot\int \frac{\diff{}{t_{n-2}}}{2\pi i} 
    \e{\pm\pi i
      t_{n-2}}\frac{\Gamma(-t_{n-2})\Gamma(b_{n-1}-a_{n}+t_{n-2})\Gamma(b_{n-2}-a_{n}+t_{n-2})}{\Gamma(b_{n-1}+b_{n-2}
      - a_n -a_{n-1} + t_{n-2})} \\
    & \phantom{=} \cdot \int \frac{\diff{}{t_{n-3}}}{2\pi i} 
    \e{\pm\pi i
      t_{n-3}}\frac{\Gamma(-t_{n-3})\Gamma(b_{n-1}+b_{n-2} - a_n + t_{n-2}-a_{n-1}+t_{n-3})}{\Gamma(b_{n-1}+b_{n-2} - a_n + t_{n-2}+b_{n-3}
      - a_{n-1} + t_{n-3})} \\
    & \phantom{=\int \frac{\diff{}{t_{n-3}}}{2\pi i} } \cdot \Gamma(b_{n-3}-a_{n-1}+t_{n-3})\\
    & \phantom{=} \cdot \,{}_{n-2}
    F_{n-3}\left(\genfrac{}{}{0pt}{}{a_1,\dots, a_{n-2}}{b_1, \dots,
        b_{n-4}, b_{n-1}+b_{n-2}+b_{n-3} - a_n -a_{n-1}+
        t_{n-2}+t_{n-3}};z\right)\\
    &=
    \frac{\Gamma(b_{n-1}) \cdots \Gamma(b_{1})}{\Gamma(a_n)
      \cdots \Gamma(a_{3})\Gamma(b_{n-1}-a_{n})\Gamma(b_{n-2}-a_{n})
      \Gamma(b_{n-3}-a_{n-1})\cdots \Gamma(b_1-a_3)}\\
    &\phantom{=} \cdot\int \frac{\diff{}{t_{n-2}}}{2\pi i} 
    \e{\pm\pi i
      t_{n-2}}\frac{\Gamma(-t_{n-2})\Gamma(b_{n-1}-a_{n}+t_{n-2})\Gamma(b_{n-2}-a_{n}+t_{n-2})}{\Gamma(b_{n-1}+b_{n-2}
      - a_n -a_{n-1} + t_{n-2})} \\
    & \phantom{=} \cdot \int \frac{\diff{}{t_{n-3}}}{2\pi i} 
    \e{\pm\pi i
      t_{n-3}}\frac{\Gamma(-t_{n-3})\Gamma(b_{n-1}+b_{n-2} - a_n
      -a_{n-1}+
      t_{n-2}+t_{n-3})}{\Gamma(b_{n-1}+b_{n-2}
      +b_{n-3}- a_n - a_{n-1} + t_{n-2}
      + t_{n-3})} \\
    & \phantom{=\int \frac{\diff{}{t_{n-3}}}{2\pi i} } \cdot \Gamma(b_{n-3}-a_{n-1}+t_{n-3})\\
    & \phantom{=} \cdot \quad \cdots\\
    & \phantom{=} \cdot \int \frac{\diff{}{t_{1}}}{2\pi i} 
    \e{\pm\pi i
      t_{1}}\frac{\Gamma(-t_{1})\Gamma(b_{n-1}+\dots + b_{2} - a_n -
      \dots - a_{3} + t_{n-2} + \dots
      +t_{1})}{\Gamma(b_{n-1}+\dots+b_{1} -
      a_n - \dots - a_{3} + t_{n-2} + \dots + t_{1})} \\
    & \phantom{=\int \frac{\diff{}{t_{n-3}}}{2\pi i} } \cdot \Gamma(b_{1}-a_{3}+t_{1})\\
    & \phantom{=} \cdot \,{}_{2}
    F_{1}\left(\genfrac{}{}{0pt}{}{a_1, a_{2}}{b_{n-1}+\dots+b_{1} -
        a_n -\dots - a_{3}+
        t_{n-2}+\dots + t_1};z\right)
  \end{align*}
Introducing new integration variables by
 \[
  u_k = \sum_{j=k}^{n-2} t_j, \qquad k=1,\dots, n-2
 \]
setting $u=u_1$ and using~\eqref{eq:c} yields the claim.
  \end{proof}

\subsection{A recurrence for $\widetilde G_p$}
\label{sec:recurr-cert-meij}

In this section we prove a new recurrence for the functions
$\widetilde G_p(z)$ that were defined in~\eqref{eq:Gtilde} and
appeared in Proposition~\ref{prop:6}. The recurrence is similar to the recurrence for ${}_nF_{n-1}(z)$ in the last section. 
\begin{lem}
  \label{lem:15}
    For any $1\leq p < n$ and if $\alpha_{p-1}+\gamma_p,
    \alpha_p+\gamma_p \not \in \mZ_{\leq 0}$, then
    \[
    \begin{aligned}
      \widetilde G_p \left(\genfrac{}{}{0pt}{0}{\alpha_{1},\dots,\alpha_p}{\gamma_{1},\dots,\gamma_p};z\right) &=\Gamma(\alpha_{p-1}+\gamma_p)\Gamma(\alpha_p+\gamma_p) \\
      &\phantom{=} \cdot \int \frac{\diff{}{s}}{2 \pi i}\e{\pm \pi i
        s}\frac{\Gamma(\alpha_{p-1}+s)\Gamma(\alpha_p+s)}{\Gamma(\alpha_{p-1}+\alpha_p+\gamma_p+s)}
      \widetilde G_{p-1}\left(\genfrac{}{}{0pt}{0}{\alpha_{1},\dots,\alpha_{p-2},-s}{\gamma_{1},\dots,\gamma_{p-1}};z\right)
    \end{aligned}
  \]
\end{lem}
\begin{proof}
  We apply Lemma~\ref{lem:5}(2) with $\lambda_1=\alpha_{p-1},\lambda_2
  = \alpha_p, \mu=t, \nu=\alpha_{p-1}+\alpha_p+\gamma_p$ to obtain
  \[
  \begin{aligned}
    \widetilde G_p
    \left(\genfrac{}{}{0pt}{0}{\alpha_{1},\dots,\alpha_p}{\gamma_{1},\dots,\gamma_p};z\right)
    &= \Gamma(\alpha_{p-1}+\gamma_p)\Gamma(\alpha_p+\gamma_p)\int
    \frac{\diff{}{t}}{2\pi i}\, z^t \prod_{h=1}^{p-2}\Gamma(\alpha_h+t) \prod_{h=1}^{p-1}
    \Gamma(\gamma_h-t) \\
    &\phantom{=} \cdot \int \frac{\diff{}{s}}{2 \pi i}\e{\pm \pi i
      s}\frac{\Gamma(t-s)\Gamma(\alpha_{p-1}+s)\Gamma(\alpha_p+s)}{\Gamma(\alpha_{p-1}+\alpha_p+\gamma_p+s)}
    \end{aligned}
  \]
  which is valid for $\Re(\gamma_p-t) > 0$. This justifies interchanging
the order of integration as in the proof of Proposition~\ref{prop:8}
and yields, after extension to the whole $t$ plane by analytic continuation of the
parameters, the desired result.
  \[
   \begin{aligned}
    \widetilde G_p
    \left(\genfrac{}{}{0pt}{0}{\alpha_{1},\dots,\alpha_p}{\gamma_{1},\dots,\gamma_p};z\right)
      &= \Gamma(\alpha_{p-1}+\gamma_p)\Gamma(\alpha_p+\gamma_p) \int \frac{\diff{}{s}}{2 \pi i}\e{\pm \pi i
      s}\frac{\Gamma(\alpha_{p-1}+s)\Gamma(\alpha_p+s)}{\Gamma(\alpha_{p-1}+\alpha_p+\gamma_p+s)}\\
    &\phantom{=} \cdot \int
    \frac{\diff{}{t}}{2\pi i}\, z^t
    \prod_{h=1}^{p-2}\Gamma(\alpha_h+t) \Gamma(-s+t) \prod_{h=1}^{p-1}
    \Gamma(\gamma_h-t) \\
  \end{aligned}\qedhere
  \]
\end{proof}

\begin{prop}
 \label{prop:10} 
  For $3 \leq p < n$, and if  $\alpha_{p}+\gamma_p,
    \alpha_s+\gamma_{s+1} \not \in \mZ_{\leq 0}$, $s=2,\dots,p-1$, then
  \[
    \widetilde G_p \left(\genfrac{}{}{0pt}{0}{\alpha_{1},\dots,\alpha_p}{\gamma_{1},\dots,\gamma_p};z\right) = \int \frac{\diff{}{s}}{2 \pi i}\, B_p(s)  \widetilde G_2 \left(\genfrac{}{}{0pt}{0}{\alpha_{1},-s}{\gamma_{1},\gamma_2};z\right)
  \]
  where
  \[
  \begin{aligned}
    B_p(s) &= \Gamma(\alpha_p+\gamma_p) \Gamma(\alpha_{p-1}+\gamma_p) \Gamma(\alpha_{p-2}+\gamma_{p-1})\dots\Gamma(\alpha_2+\gamma_3) \\
   &\phantom{=} \cdot\int \frac{\diff{}{s_{p-2}}}{2\pi i}\, \e{-i\pi s_{p-2}} \frac{\Gamma(\alpha_p+s_{p-2})\Gamma(\alpha_{p-1}+s_{p-2})\Gamma(\gamma_{p-1}+s_{p-2})}{\Gamma(\alpha_p+\alpha_{p-1}+\gamma_p+s_{p-2})} \\
   &\phantom{=} \cdot\int \frac{\diff{}{s_{p-3}}}{2\pi i}\, \e{-i\pi s_{p-3}} \frac{\Gamma(\alpha_{p-2}+s_{p-3})\Gamma(\gamma_{p-2}+s_{p-3})\Gamma(-s_{p-2}+s_{p-3})}{\Gamma(\alpha_{p-2}+\gamma_{p-1}-s_{p-2}+s_{p-3})} \\
   &\phantom{=}\cdot \phantom{int} \dots\\
   &\phantom{=} \cdot\int \frac{\diff{}{s_{2}}}{2\pi i}\, \e{-i\pi s_{2}} \frac{\Gamma(\alpha_{3}+s_{2})\Gamma(\gamma_{3}+s_{2})\Gamma(-s_{3}+s_{2})}{\Gamma(\alpha_{3}+\gamma_{4}-s_{3}+s_{2})} \\
  &\phantom{=} \cdot \e{-i\pi s} \frac{\Gamma(\alpha_2+s)\Gamma(-s_2+s)}{\Gamma(\alpha_2+\gamma_3-s_2+s)}\\
  \end{aligned}
  \]
\end{prop}
\begin{proof}
  Repeated application of Lemma~\ref{lem:15} immediately yields the result.
\end{proof}
For the applications in Section~\ref{sec:applications} we give the
first case explicitly.
\begin{equation}
  \label{eq:B3}
   B_3(s) = \Gamma(\alpha_3+\gamma_3) \Gamma(\alpha_2+\gamma_3)
   \e{-i\pi s}
   \frac{\Gamma(\alpha_2+s)\Gamma(\alpha_3+s)}{\Gamma(\alpha_2+\alpha_3+\gamma_3+s)}
\end{equation}

\section{Analytic continuation to $z=1$ for $n>2$}
\label{sec:analyt-cont-z=1}

In this section we discuss the analytic continuation from $z=0$ to
$z=1$ for the solutions to the hypergeometric differential equation of
arbitrary order $n>2$. We first review the known results for the special
solution $\bar \xi_n(z)$ and the holomorphic solutions
${}_nF_{n-1}(z)$, then we prove the theorem for the logarithmic
solutions $G_p(z)$.

\subsection{The solution $\xi_n$}
\label{sec:solution-xi_n-1}

For completeness and application in Section~\ref{sec:applications} we
reproduce here N\o rlunds beautiful formula for the analytic
continuation of the special solution $\xi_n$. 
\begin{prop}
  \label{prop:16}
  \begin{enumerate}
  \item If $\gamma_1,\dots,\gamma_n$ are not in resonance, then
    \[
    \xi_n(z) = \Gamma(\beta_n+1) \sum_{j=1}^n
    \frac{\prod_{\substack{k=1,k\not = j}}^n \Gamma( \gamma_k -
      \gamma_j )}{ \prod_{k=1}^n \Gamma(1-\alpha_k-\gamma_j ) }
    y_j^*(z)
    \]
    \item If $\gamma_1,\dots,\gamma_q$ are in resonance, then
    \[
    \begin{aligned}
      \xi_n(z) &= \frac{\Gamma(\beta_n+1)}{2\pi i} \sum_{j=1}^q
      \frac{(-1)^j}{(q-j)!} \psi^{(q-j)}(\e{2\pi i \gamma_1}) \e{-2\pi
        i
        j \gamma_1} y_j^*(z) \\
      &\phantom{=} + \frac{1}{\pi} \sum_{j=q+1}^n
      \frac{\prod_{\substack{k=1,k\not = j}}^n \Gamma( \gamma_k -
        \gamma_j )}{ \prod_{k=1}^n \Gamma(1-\alpha_k-\gamma_j ) }
      y_j^*(z)\\
    \end{aligned}
    \]
    where
    \[
    \psi(x) = \e{-i\pi \beta_n} \frac{\prod_{k=1}^n (x-\e{-2\pi i
        \alpha_k})}{\prod_{k=q+1}^n (x-\e{2\pi i \gamma_k})}
    \]
  \end{enumerate}
\end{prop}
\begin{proof}
  We are not aware of a simple proof involving only integral
  representations. Therefore, we refer to the
  proof of~\cite{Norlund:1955ab} based on partial fraction decompositions
  and the special properties of the functions $\xi_n(z)$.
\end{proof}

\subsection{Holomorphic solutions}
\label{sec:holom-solut}

For completeness and application in Section 6 we reproduce here part
of the result~\cite{Buehring:1992ab} for the analytic continuation of
the holomorphic solutions ${}_{n}F_{n-1}(z)$ in terms of integral
representations. The complete result will be given in
Section~\ref{sec:evaluation-integrals} when we evaluate the following
integral. In order to state the result, recall the value of $c$
from~\eqref{eq:c}. 
\begin{prop}
  \label{prop:17}
    \[
  \begin{aligned}
  & \frac{\Gamma(a_1)\cdots\Gamma(a_n)}{\Gamma(b_1)\cdots\Gamma(b_{n-1})} \,{}_n F_{n-1}\left(\genfrac{}{}{0pt}{}{a_1,\dots,a_n}{b_1,\dots, b_{n-1}};z\right) \\
  & = \int \frac{\diff{}{t}}{2\pi i}\int \frac{\diff{}{u}}{2\pi i} \widetilde A^{(n)}(t)\frac{\Gamma(a_1+u)\Gamma(a_2+u)\Gamma(c+t-u)\Gamma(-u)}{\Gamma(c+a_2+t)\Gamma(c+a_1+t)} (1-z)^u 
  \end{aligned}
  \]
\end{prop}
\begin{proof}
  This follows immediately from Propositions~\ref{prop:13} and~\ref{prop:8}.
\end{proof}

\subsection{Logarithmic solutions}
\label{sec:logar-solut}

In this section we start to prove the main theorem. We give an integral
representation for the solutions $G_p(z)$, $p>1$ which admits an
easy power series expansion about $z=1$. This expansion will be
discussed in Section~\ref{sec:evaluation-integrals}. If
$\{\gamma_1,\dots,\gamma_q\}$, $q>1$, is resonant, then, as discussed in
Section~\ref{sec:hyperg-diff-equat}, have logarithms near $z=0$.
\begin{prop}
  \label{thm:12}
  For $2 < p \leq q \leq n$, and if $\Re \beta_n > \Re\beta_p$,
 $\Re(\alpha_s+\gamma_j)>0$, $j=1\dots,p$, $s=p+1,\dots,n$, $\alpha_{p}+\gamma_p,
    \alpha_s+\gamma_{s+1} \not \in \mZ_{\leq 0}$, $s=2,\dots,p-1$ then
  \[
    \begin{aligned}
    G_p(z) &= \Gamma(\alpha_1+\gamma_2) \int \frac{\diff{}{v}}{2\pi i} \e{-i\pi v} \Gamma(\alpha_1+\gamma_1+v) \Gamma(-v) \\
   &\phantom{=} \cdot \int \frac{\diff{}{s}}{2 \pi i} B_p(s) \frac{ \Gamma(\gamma_2-s) \Gamma(\gamma_1+v-s) }{\Gamma(\alpha_1+\gamma_1+\gamma_2+v-s)}\\
   &\phantom{=} \cdot 
    z^{\gamma_1} \int
    \frac{\diff{}{u}}{2\pi i} \e{-i\pi u} z^u
    \frac{\Gamma(-v+u)\Gamma(-u)}{\Gamma(-v)} \prod_{s=p+1}^n
    \frac{\Gamma(\alpha_s+\gamma_1+u)}{\Gamma(1-\gamma_s+\gamma_1+u)}
  \end{aligned}
  \]
  If $p=2$, then
  \[
    \begin{aligned}
    G_2(z) &= \Gamma(\alpha_1+\gamma_2)\Gamma(\alpha_2+\gamma_2) \int \frac{\diff{}{v}}{2\pi i} \e{-i\pi v} \frac{\Gamma(\alpha_1+\gamma_1+v) \Gamma(\alpha_2+\gamma_1+v)  \Gamma(-v)}{\Gamma(\alpha_1+\alpha_2+\gamma_1+\gamma_2+v)} \\
   &\phantom{=} \cdot 
    z^{\gamma_1} \int
    \frac{\diff{}{u}}{2\pi i} \e{-i\pi u} z^u
    \frac{\Gamma(-v+u)\Gamma(-u)}{\Gamma(-v)} \prod_{s=3}^n
    \frac{\Gamma(\alpha_s+\gamma_1+u)}{\Gamma(1-\gamma_s+\gamma_1+u)}
  \end{aligned}
  \]
\end{prop}
\begin{proof}
  We first consider the case $p>2$.
  We abbreviate the left--hand side by $G_p(z)$. By
  Proposition~\ref{prop:6} we can have
  \[
    \begin{aligned}
    G_p(z) &= \frac{1}{\Gamma(\beta_n-\beta_p)}
    \int_0^z \frac{\diff{}{t}}{t}\, \widetilde G_p \left(\genfrac{}{}{0pt}{0}{\alpha_{1},\dots,\alpha_p}{\gamma_{1},\dots,\gamma_p};t\right)\bar\xi_{n-p}\left(\genfrac{}{}{0pt}{0}{\alpha_{p+1},\dots,\alpha_n}{\gamma_{p+1},\dots,\gamma_n};\frac{z}{t}\right)\\
  \end{aligned}
  \]
  By Proposition~\ref{prop:10} we can write this as
  \[
    \begin{aligned}
    G_p(z) &= \frac{1}{\Gamma(\beta_n-\beta_p)}
    \int_0^z \frac{\diff{}{t}}{t}\, \bar\xi_{n-p}\left(\genfrac{}{}{0pt}{0}{\alpha_{p+1},\dots,\alpha_n}{\gamma_{p+1},\dots,\gamma_n};\frac{z}{t}\right) \int \frac{\diff{}{s}}{2 \pi i} B_p(s) G_2 \left(\genfrac{}{}{0pt}{0}{\alpha_{1},-s}{\gamma_{1},\gamma_2};t\right)
  \end{aligned}
  \]
  Shifting the integration variable in $G_2$ yields
  \[
G_2 \left(\genfrac{}{}{0pt}{0}{\alpha_{1},-s}{\gamma_{1},\gamma_2};t\right) = t^{\gamma_1} G_2 \left(\genfrac{}{}{0pt}{0}{\alpha_{1}+\gamma_1,\gamma_1-s}{\gamma_{2}-\gamma_1,0};t\right)
  \]
  By Proposition~\ref{prop:8} with $a=\alpha_1+\gamma_1, b=\gamma_1-s, c=\alpha_1+\gamma_1+\gamma_2-s$ and using the integral representation of ${}_2F_1$ we obtain
  \[
    \begin{aligned}
    G_p(z) &= \frac{\Gamma(\alpha_1+\gamma_2) }{\Gamma(\beta_n-\beta_p)}
    \int_0^z \frac{\diff{}{t}}{t}\,t^{\gamma_1} \bar\xi_{n-p}\left(\genfrac{}{}{0pt}{0}{\alpha_{p+1},\dots,\alpha_n}{\gamma_{p+1},\dots,\gamma_n};\frac{z}{t}\right) \int \frac{\diff{}{s}}{2 \pi i} \,B_p(s) \Gamma(\gamma_2-s) \\
   &\phantom{=} \cdot \int \frac{\diff{}{v}}{2\pi i} (1-t)^v \e{-i\pi v} \frac{\Gamma(\alpha_1+\gamma_1+v)\Gamma(\gamma_1-s+v)\Gamma(-v)}{\Gamma(\alpha_1+\gamma_1+\gamma_2-s+v)} 
  \end{aligned}
  \]
Under the conditions on $\alpha_i,\beta_j,\gamma_k$ stated in the
claim, we can change the order of the integrals
  \[
    \begin{aligned}
    G_p(z) &= \frac{\Gamma(\alpha_1+\gamma_2)}{\Gamma(\beta_n-\beta_p)} \int \frac{\diff{}{v}}{2\pi i} \e{-i\pi v} \Gamma(\alpha_1+\gamma_1+v) \Gamma(-v) \\
   &\phantom{=} \cdot \int \frac{\diff{}{s}}{2 \pi i} B_p(s) \frac{\Gamma(\gamma_2-s) \Gamma(\gamma_1+v-s) }{\Gamma(\alpha_1+\gamma_1+\gamma_2+v-s)}\\ 
   &\phantom{=} \cdot 
    \int_0^z \frac{\diff{}{t}}{t}\, t^{\gamma_1} (1-t)^v  \bar\xi_{n-p}\left(\genfrac{}{}{0pt}{0}{\alpha_{p+1},\dots,\alpha_n}{\gamma_{p+1},\dots,\gamma_n};\frac{z}{t}\right) 
  \end{aligned}
  \]
Finally, the $t$ integral can be evaluated by Lemma~\ref{lem:2}. This
yields the claim.

Now consider the case $p=2$. The first step using
Proposition~\ref{prop:6} for general $p$ applies here,
too. However, for the second step we do not need the recursion formula
from Proposition~\ref{prop:10}. The remaining steps are as above. We
shift again the integration variable in $G_2$ and use
Proposition~\ref{prop:8} to get
  \[
    \begin{aligned}
    G_2(z) &= \frac{\Gamma(\alpha_1+\gamma_2)\Gamma(\alpha_2+\gamma_2)}{\Gamma(\beta_n-\beta_2)}
    \int_0^z \frac{\diff{}{t}}{t}\,t^{\gamma_1} \bar\xi_{n-2}\left(\genfrac{}{}{0pt}{0}{\alpha_{3},\dots,\alpha_n}{\gamma_{3},\dots,\gamma_n};\frac{z}{t}\right) \\
   &\phantom{=} \cdot \int \frac{\diff{}{v} }{2\pi i} \,(1-t)^v\e{-i\pi v} \frac{\Gamma(\alpha_1+\gamma_1+v) \Gamma(\alpha_2+\gamma_1+v)  \Gamma(-v)}{\Gamma(\alpha_1+\alpha_2+\gamma_1+\gamma_2+v)} \\
  \end{aligned}
  \]
After changing the order of integration and evaluating the $t$ integral
with the help of Lemma~\ref{lem:2} we obtain the claim.
\end{proof}

  \begin{rem}
    \label{rem:24}
    By~\eqref{eq:nonresonantyj} the $u$ integral in Theorem~\ref{thm:12} is nothing but
  the integral representation of a hypergeometric function:
  \[
  \begin{aligned}
    &z^{\gamma_1} \int \frac{\diff{}{u}}{2\pi i} \e{-i\pi u} z^u
    \frac{\Gamma(-v+u)\Gamma(-u)}{\Gamma(-v)} \prod_{s=p+1}^n
    \frac{\Gamma(\alpha_s+\gamma_1+u)}{\Gamma(1-\gamma_s+\gamma_1+u)}\\
    &= z^{\gamma_1} \prod_{k=p+1}^n
    \frac{\Gamma(\alpha_k+\gamma_1)}{\Gamma(1-\gamma_k+\gamma_1)} 
      {}_{n-p+1}F_{n-p}\left(\genfrac{}{}{0pt}{0}{-v,\alpha_{p+1}+\gamma_1,\dots,\alpha_n+\gamma_1}{1-\gamma_{p+1}+\gamma_1,\dots,1-\gamma_n+\gamma_1};z\right)
  \end{aligned}
  \]   
  \end{rem}

\section{Series expansions}
\label{sec:series-expansions}

\subsection{Evaluation of the integrals}
\label{sec:evaluation-integrals}

Ultimately, we want to compare the analytic continuation of the
integral representation with the series solutions obtained from the
Frobenius method. To achieve this we have to evaluate the various
integrals that we obtained in Section 5. We formulate the results as
corollaries.

For this purpose we define the numbers $A^{(n)}(k)$ by taking the
residue of the functions $\widetilde A^{(n)}(u)$ given in
Proposition~\ref{prop:13} at $u=k, k \in \mN$:
  \[
  A^{(n)}(k) := \Res_{u=k} \widetilde A^{(n)}(u).
  \]
Of course, the internal integrals of $A^{(n)}(u)$ for $n>3$ can be
evaluated with the residue theorem. For later purposes we give
the result for $n=3$ and $n=4$ obtained from~\eqref{eq:A3t}
and~\eqref{eq:A4t}, respectively:
\begin{align}
  \label{eq:A3}
  A^{(3)}(k) &=
  \frac{\Gamma(b_2-a_3+k)\Gamma(b_1-a_3+k)}{\Gamma(b_2-a_3)\Gamma(b_1-a_3)\Gamma(k+1)}\\
  \label{eq:A4}
   A^{(4)}(k) &= \frac{\Gamma(b_{3}+b_{2}
      - a_4 - a_{3} + k)}{\Gamma(b_{3}-a_{4})\Gamma(b_{2}-a_{4})
      \Gamma(b_1-a_3)} \\
    &\phantom{=} \cdot \sum_{k_2=0}^k\frac{\Gamma(b_{1}-a_{3}+k-k_2) \Gamma(b_{3} - a_4 +
      k_{2})\Gamma(b_{2}-a_{4}+k_{2}) }{\Gamma(k-k_2+1)\Gamma(b_{3}+b_{2}
      - a_4 - a_{3} + k_{2}) \Gamma(k_2+1)}\notag
\end{align}

As a corollary to Proposition~\ref{prop:17} we obtain the following
series expansions in powers of $1-z$ of ${}_nF_{n-1}(z)$ due to
B\"uhring~\cite{Buehring:1992ab}. They strongly depend on the
integrality of the number $c$ in~\eqref{eq:c}. 
\begin{cor}
  \label{cor:20}
  If $c \not \in \mZ$ then
  \[
    \begin{aligned}
      &
      \frac{\Gamma(a_1)\cdots\Gamma(a_n)}{\Gamma(b_1)\cdots\Gamma(b_{n-1})}
      \,{}_n
      F_{n-1}\left(\genfrac{}{}{0pt}{}{a_1,\dots,a_n}{b_1,\dots,
          b_{n-1}};z\right) \\
      & =\sum_{m=0}^\infty g_m(0)(1-z)^m + (1-z)^c \sum_{m=0}^\infty
      g_m(c)(1-z)^m
    \end{aligned}
  \]
  where
  \[
  \begin{aligned}
    g_m(\ell) &= (-1)^m
    \frac{\Gamma(a_1+\ell+m)\Gamma(a_2+\ell+m)\Gamma(c-2\ell-m)}{\Gamma(c+a_1)\Gamma(c+a_2)\Gamma(m+1)}\\
    &\phantom{=} \cdot \sum_{k=0}^\infty
    \frac{(c-\ell-m)_k}{(c+a_1)_k(c+a_2)_k}A^{(n)}(k)
  \end{aligned}
  \]
  The series in $g_m(\ell)$ terminates when $\ell =c$, while for $\ell=0$ we need the
conditions $\Re(a_j+m) > 0$, $j=3,\dots,n$, for convergence.
\end{cor}
\begin{proof}
  We evaluate the integral on the right-hand side of Proposition~\ref{prop:17}. 
  \[
  \int \frac{\diff{}{t}}{2\pi i}\int \frac{\diff{}{u}}{2\pi i}
  \widetilde
  A^{(n)}(t)\frac{\Gamma(a_1+u)\Gamma(a_2+u)\Gamma(c+t-u)\Gamma(-u)}{\Gamma(c+a_2+t)\Gamma(c+a_1+t)}
  (1-z)^u 
  \]
  Since $c$ is not an integer, we first carry out the integral over $t$,
  then the integral over $u$. The integral over $t$ is simply obtained
  by closing the contour to the right
  \[
    \int \frac{\diff{}{t}}{2\pi i } \widetilde A^{(n)}(t)
    \frac{ \Gamma(c+t-u)}{\Gamma(c+a_2+t)\Gamma(c+a_1+t)}
    = \sum_{k\geq 0} A_n(k) \frac{
      \Gamma(c+k-u)}{\Gamma(c+a_2+k)\Gamma(c+a_1+k)}\\
  \]
  The remaining integral over $u$ is a $G_2$ with lower parameters
  $c+k$ and $0$ which, by assumption, are not in resonance. Hence we
  can apply Lemma~\ref{lem:4}(1) and the reflection formula for the
  Gamma function.
  \[
  \begin{aligned}
    &\int
    \frac{\diff{}{u}}{2\pi i
    }\Gamma(a_1+u)\Gamma(a_2+u)\Gamma(c+k-u)\Gamma(-u) (1-z)^u\\
    &= \Gamma(c+k)\Gamma(1-c-k)\int_{-i\infty}^{i\infty}
    \frac{\diff{}{u}}{2\pi
      i}\frac{\Gamma(a_1+u)\Gamma(a_2+u)\Gamma(-u)}{\Gamma(1-c-k+u)}
      (z-1)^u\\
    &\phantom{=}+ (1-z)^{c+k}\Gamma(-c-k)\Gamma(1+c+k)\\
    &\phantom{=} \cdot \int
    \frac{\diff{}{u}}{2\pi
      i}\frac{\Gamma(c+a_1+k+u)\Gamma(c+a_2+k+u)\Gamma(-u)}{\Gamma(1+c+k+u)}
    (z-1)^u\\
  \end{aligned}
   \]
  Shifting the integration variable in the second integral and closing
  the contour to the right in both integrals, this expression becomes
   \[
  \begin{aligned}
    &\Gamma(c+k)\Gamma(1-c-k) \sum_{m=0}^\infty \frac{\Gamma(a_1+m)\Gamma(a_2+m)}{\Gamma(1-c-k+m)\Gamma(m+1)}
      (1-z)^m\\
    &\phantom{=}+
    (1-z)^{c}\Gamma(-c-k)\Gamma(1+c+k) \sum_{m=0}^\infty \frac{\Gamma(c+a_1+m)\Gamma(c+a_2+m)}{\Gamma(1+c+m)\Gamma(m-k+1)}  (1-z)^m\\
  \end{aligned}
  \]
  Hence, combining everything we find
  \[
  \begin{aligned}
    & \frac{\Gamma(a_1)\cdots\Gamma(a_n)}{\Gamma(b_1)\cdots\Gamma(b_{n-1})}
    \,{}_n F_{n-1}\left(\genfrac{}{}{0pt}{}{a_1,\dots,a_n}{b_1,\dots,
        b_{n-1}};z\right) \\
    &= \sum_{k=0}^\infty \frac{A^{(n)}(k)\Gamma(c+k)\Gamma(1-c-k)}{\Gamma(c+a_2+k)\Gamma(c+a_1+k)}
    \sum_{m=0}^\infty \frac{\Gamma(a_1+m)\Gamma(a_2+m)}{\Gamma(1-c-k+m)\Gamma(m+1)}
      (1-z)^m\\
    &\phantom{=} +\frac{A^{(n)}(k)
    \Gamma(-c-k)\Gamma(1+c+k)}{\Gamma(c+a_2+k)\Gamma(c+a_1+k)}
  (1-z)^{c}\sum_{m=0}^\infty  
    \frac{\Gamma(c+a_1+m)\Gamma(c+a_2+m)}{\Gamma(1+c+m)\Gamma(m-k+1)}
    (1-z)^m
  \end{aligned}
  \]
  We massage this expression by using
  $\Gamma(\alpha+\ell)\Gamma(1-\alpha-\ell) =
  (-1)^\ell\Gamma(\alpha)\Gamma(1-\alpha)$ 
  for $\alpha\in\mC$ and $\ell\in\mZ$ three times: With $\alpha=c+k$,
  $\ell=m$ in the first summand, and with $\alpha=-c$, $\ell=m+k$, as
  well as $\alpha=1+m$, $\ell=-k$ in the second summand. The
  convergence of the sums over $m$ for 
  $\Re(a_j+m)>0$, $j=3,\dots,n$, allows us to interchange the sums. This yields
  \begin{align}
    & \sum_{k=0}^\infty \frac{A^{(n)}(k)}{\Gamma(c+a_2+k)\Gamma(c+a_1+k)}
    \sum_{m=0}^\infty
    (-1)^m\frac{\Gamma(a_1+m)\Gamma(a_2+m)\Gamma(c+k-m)}{\Gamma(m+1)} (1-z)^m\notag\\
    &\phantom{=} +\frac{A^{(n)}(k)}{\Gamma(c+a_2+k)\Gamma(c+a_1+k)}\notag\\
    &\phantom{=} \cdot 
  (1-z)^{c}\sum_{m=0}^\infty  (-1)^{m+k}
    \frac{\Gamma(c+a_1+m)\Gamma(c+a_2+m)\Gamma(-c-m)}{\Gamma(m-k+1)}
    (1-z)^m\notag\\
    &= \sum_{m=0}^\infty(-1)^m  \frac{\Gamma(a_1+m)\Gamma(a_2+m)}{\Gamma(m+1)}\sum_{k=0}^\infty \frac{A^{(n)}(k)\Gamma(c+k-m)}{\Gamma(c+a_2+k)\Gamma(c+a_1+k)}
      (1-z)^m\notag\\ 
    &\phantom{=} +
    (1-z)^{c}\sum_{m=0}^\infty (-1)^m \frac{\Gamma(-m-c)\Gamma(c+a_1+m)\Gamma(c+a_2+m)}{\Gamma(1+m)\Gamma(-m)} \notag\\ 
    &\phantom{=} \cdot \sum_{k=0}^\infty (-1)^k \frac{A^{(n)}(k)\Gamma(-m+k)}{\Gamma(c+a_2+k)\Gamma(c+a_1+k)}  (1-z)^{m}\notag
  \end{align}
  Rewriting this expression in terms of Pochhammer symbols yields the claim.
\end{proof}
For the applications in Section~\ref{sec:applications} we need the
case $n=3$ explicitly. 
In particular,
\begin{equation}
  \label{eq:g0_c}
  g_0(c) = \Gamma(-c).
\end{equation}
in agreement with N\o rlund's analysis~\cite{Norlund:1955ab}. Since
the indices of the hypergeometric differential equation at $z=1$ are
$0,1$ and $c=b_1+b_2-a_1-a_2-a_3$, it is sufficient to determine the
coefficients $g_{j}(0)$ for $j=0, 1$ besides
$g_0(c)$. With~\eqref{eq:A3} we find
\begin{equation}
  \begin{aligned}
   \label{eq:g0_0}  
    g_0(0) &= \frac{\Gamma(a_1)\Gamma(a_2)\Gamma(c)}{\Gamma(c+a_1)\Gamma(c+a_2)}\sum_{k=0}^\infty \frac{(c)_k(b_2-a_3)_k(b_1-a_3)_k}{(a_1+c)_k(a_2+c)_k k!} \\
    &= \frac{\Gamma(a_1)\Gamma(a_2)\Gamma(c)}{\Gamma(c+a_1)\Gamma(c+a_2)} \,{}_3F_2\left(\genfrac{}{}{0pt}{}{c, b_2-a_3, b_1-a_3}{a_1+c,a_2+c};1\right)
  \end{aligned}
\end{equation}
and
\begin{equation}
  \begin{aligned}
    \label{eq:g1_0}  
    g_1(0) &= -\frac{\Gamma(a_1+1)\Gamma(a_2+1)\Gamma(c-1)}{\Gamma(c+a_1)\Gamma(c+a_2)}\sum_{k=0}^\infty \frac{(c-1)_k(b_2-a_3)_k(b_1-a_3)_k}{(a_1+c)_k(a_2+c)_k k!} \\
    &= -\frac{\Gamma(a_1+1)\Gamma(a_2+1)\Gamma(c-1)}{\Gamma(c+a_1)\Gamma(c+a_2)} \,{}_3F_2\left(\genfrac{}{}{0pt}{}{c-1, b_2-a_3, b_1-a_3}{a_1+c,a_2+c};1\right)
  \end{aligned}
\end{equation}

If $c$ is a nonnegative integer, Proposition~\ref{prop:17} yields an entirely different series
expansion, again due to~\cite{Buehring:1992ab}.
\begin{cor}
  \label{cor:21}
  If $c=c_0 \in \mZ_{\geq 0}$ then
  \[
  \begin{aligned}
    &
    \frac{\Gamma(a_1)\cdots\Gamma(a_n)}{\Gamma(b_1)\cdots\Gamma(b_{n-1})}
    \,{}_n F_{n-1}\left(\genfrac{}{}{0pt}{}{a_1,\dots,a_n}{b_1,\dots,
        b_{n-1}};z\right) \\
    & =\sum_{m=0}^{{c_0}-1} l_m(1-z)^m + (1-z)^{c_0} \sum_{m=0}^\infty \left(
      w_m + q_m\log(1-z) \right)(1-z)^m
  \end{aligned}
  \]
  with
  \[
  \begin{aligned}
    l_m &= g_m(0)|_{c={c_0}} = (-1)^m
    \frac{\Gamma(a_1+m)\Gamma(a_2+m)\Gamma({c_0}-m)}{\Gamma({c_0}+a_1)\Gamma({c_0}+a_2)\Gamma(m+1)}\sum_{k=0}^\infty
    \frac{({c_0}-m)_k}{(a_1+{c_0})_k(a_2+{c_0})_k}A^{(n)}(k)\\
    q_m &= g_m(c)|_{c={c_0}} =
    (-1)^{{c_0}+1}\frac{(a_1+{c_0})_m(a_2+{c_0})_m}{\Gamma({c_0}+m+1)\Gamma(m+1)}\sum_{k=0}^m
    \frac{(-m)_k}{(a_1+{c_0})_k(a_2+{c_0})_k}A^{(n)}(k)\\
    w_m&=
    (-1)^{{c_0}}\frac{(a_1+{c_0})_m(a_2+{c_0})_m}{\Gamma({c_0}+m+1)\Gamma(m+1)}\sum_{k=0}^m
    \frac{(-m)_k}{(a_1+{c_0})_k(a_2+{c_0})_k}A^{(n)}(k)\notag\\
    &\phantom{=} \cdot \left( \psi(1+m-k) + \psi(1+{c_0}+m) -
      \psi(a_1+{c_0}+m)
      - \psi(a_2+{c_0}+m) \right)\\
    &\phantom{=} +
    (-1)^{{c_0}+m}\frac{(a_1+{c_0})_m(a_2+{c_0})_m}{\Gamma({c_0}+m+1)}\sum_{k=m+1}^\infty
    \frac{\Gamma(k-m)}{(a_1+{c_0})_k(a_2+{c_0})_k}A^{(n)}(k)\notag
  \end{aligned}
  \]
  The convergence of the series in $l_m$ requires the
conditions $\Re(a_j+m) > 0$, $j=3,\dots,n$, while whe convergence of
the series in $w_m$ requires the conditions $\Re(c_0+a_j+m) > 0$, $j=3,\dots,n$.
\end{cor}
\begin{proof}
  We begin as in the proof of Corollary~\ref{cor:20} and evaluate the
  $t$ integral in Proposition~\ref{prop:17}. The remaining integral
  over $u$ is a $G_2$ with lower 
  parameters $c + k$ and $0$ which, by assumption, are now in
  resonance. Instead of applying Lemma~\ref{lem:4}(1) we have to
  evaluate the residue integral explicitly by closing the contour to the right
  and picking up the double poles. This is a lengthy but
  straightforward computation whose result can be found
  e.g. in~\cite {Erdelyi:1953ab} or in~\cite{Norlund:1963ab}. It reads
  in our case
  \[
  \begin{aligned}
    &(-1)^{c_0+k}\int
    \frac{\diff{}{u}}{2\pi i
    }\Gamma(a_1+u)\Gamma(a_2+u)\Gamma(c_0+k-u)\Gamma(-u) (1-z)^u\\
    & =  - \log
    (1-z) (1-z)^{c_0+k}\sum_{m\geq 0}
     \frac{\Gamma(a_1+c_0+k+m)\Gamma(a_2+c_0+k+m)}{\Gamma(1+c_0+k+m)\Gamma(m+1)}(1-z)^m\\
     &\phantom{=} - (1-z)^{c_0+k} \sum_{m=0}^\infty
     \frac{\Gamma(a_1+c_0+k+m)\Gamma(a_2+c_0+k+m)}{\Gamma(1+c_0+k+m)\Gamma(m+1)}\\
     &\phantom{=}
    \cdot \left(\psi(a_1+c_0+k+m)+\psi(a_2+c_0+k+m)-\psi(1+c_0+k+m)-\psi(1+m)\right) (1-z)^m\\
    &\phantom{=} + (-1)^{c_0+k}\Gamma(c_0+k) \Gamma(1-c_0-k)\sum_{m=0}^{c_0+k-1}  \frac{\Gamma(a_1+m)\Gamma(a_2+m)}{\Gamma(1-c_0-k+m) \Gamma(m+1)}(1-z)^m
  \end{aligned}
  \]
  where $\psi(z) = (\log \Gamma(z) )'$ is the Digamma function. Hence,
  combining everything we find
  \begin{align*}
    & \frac{\Gamma(a_1)\cdots\Gamma(a_n)}{\Gamma(b_1)\cdots\Gamma(b_{n-1})}
    \,{}_n F_{n-1}\left(\genfrac{}{}{0pt}{}{a_1,\dots,a_n}{b_1,\dots,
        b_{n-1}};z\right) \notag\\
    &= \sum_{k=0}^\infty
    \frac{(-1)^{c_0+k}A^{(n)}(k)}{\Gamma(c_0+a_2+k)\Gamma(c_0+a_1+k)}\left(  \vphantom{\sum_{m=0}^\infty\frac{\Gamma(k)}{\Gamma(k)}}- \log
    (1-z) (1-z)^{c_0+k}\right.\notag\\
     &\phantom{=} \cdot \sum_{m=0}^\infty
     \frac{\Gamma(a_1+c_0+k+m)\Gamma(a_2+c_0+k+m)}{\Gamma(1+c_0+k+m)\Gamma(m+1)}(1-z)^m\notag\\
     &\phantom{=} - (1-z)^{c_0+k} \sum_{m= 0}^\infty
     \frac{\Gamma(a_1+c_0+k+m)\Gamma(a_2+c_0+k+m)}{\Gamma(1+c_0+k+m)\Gamma(m+1)}\notag\\
     &\phantom{=}
    \cdot \left(\psi(a_1+c_0+k+m)+\psi(a_2+c_0+k+m)-\psi(1+c_0+k+m)-\psi(1+m)\right) (1-z)^m\notag\\
    &\phantom{=} \left.+ (-1)^{c_0+k}\Gamma(c_0+k)\Gamma(1-c_0-k)\sum_{m=0}^{c_0+k-1}
      \frac{\Gamma(a_1+m)\Gamma(a_2+m)}{\Gamma(1-c_0-k+m)
      \Gamma(m+1)}(1-z)^m\right)\notag\\
  \end{align*}
  For the first two summands in the parenthesis a standard argument in
  the theory of hypergeometric functions shows that the
  infinite sums over $m$ converge for $\Re(a_j+m)>0$, $j=3,\dots,n$,
  hence we can interchange the sums 
  over $m$ and $k$. In the third summand we split the sum over $m$ into a
  sum from $0$ to $c_0-1$ and a sum over the remaining values of $m$. In
  the latter we shift the summation index. Then this  
  expression on the right hand side becomes 
  \begin{align*}
    &(1-z)^{c_0} \log(1-z) \sum_{m=0}^\infty (-1)^{c_0+1} 
      \frac{\Gamma(a_1+c_0+m)\Gamma(a_2+c_0+m)}{\Gamma(1+c_0+m)}\notag\\
    & \cdot
      \sum_{k=0}^{m} \frac{ (-1)^k A^{(n)}(k)
      }{\Gamma(a_1+c_0+k)\Gamma(a_2+c_0+k)\Gamma(1+m-k)}
      (1-z)^m\notag\\
    & +  (1-z)^{c_0}\sum_{m=0}^\infty (-1)^{c_0+1} 
      \frac{\Gamma(a_1+c_0+m)\Gamma(a_2+c_0+m)}{\Gamma(1+c_0+m)}\notag\\
    & \cdot
      \sum_{k=0}^{m} \frac{ (-1)^k A^{(n)}(k)
      }{\Gamma(a_1+c_0+k)\Gamma(a_2+c_0+k)\Gamma(1+m-k)}\notag\\
     & \cdot
       \left(\psi(a_1+c_0+m)+\psi(a_2+c_0+m)-\psi(1+c_0+m)-\psi(1+m-k)\right)  
      (1-z)^m\notag\\
     &+\sum_{k=0}^\infty
       \frac{A^{(n)}(k)\Gamma(c_0+k) \Gamma(1-c_0-k)}{\Gamma(a_1+c_0+k)\Gamma(a_2+c_0+k)}\sum_{m=0}^{c_0-1} \frac{\Gamma(a_1+m)\Gamma(a_2+m)}{\Gamma(1-c_0-k+m)\Gamma(m+1)}(1-z)^m\notag\\
     &+(1-z)^{c_0}\sum_{k=0}^\infty
       \frac{A^{(n)}(k)\Gamma(c_0+k)\Gamma(1-c_0-k)}{\Gamma(a_1+c_0+k)\Gamma(a_2+c_0+k)}\sum_{m=0}^{k}
       \frac{\Gamma(a_1+c_0+m)\Gamma(a_2+c_0+m)}{\Gamma(1-k+m)\Gamma(1+c_0+m)}(1-z)^m\notag\\ 
   \end{align*}
  Again, we massage this expression by using
  $\Gamma(\alpha+\ell)\Gamma(1-\alpha-\ell) =
  (-1)^\ell\Gamma(\alpha)\Gamma(1-\alpha)$ 
  for $\alpha\in\mC$ and $\ell\in\mZ$ four times: With $\alpha=1+m$,
  $\ell=-k$ in the first and second summand, with $\alpha=c_0+k$,
  $\ell=m$ in the third summand, and with $\alpha=c_0+k$,
  $\ell=-m-c_0$. Furthermore, if we require that $\Re(c_0+a_j+m)>0$,
  $j=3,\dots,n$, we can interchange the sums over $m$ and $k$
  in the fourth summand. This yields
   \begin{align*}
    &(1-z)^{c_0} \log(1-z) \sum_{m=0}^\infty (-1)^{c_0+1} 
      \frac{\Gamma(a_1+c_0+m)\Gamma(a_2+c_0+m)}{\Gamma(1+c_0+m) \Gamma(m+1)}\\
    & \cdot
      \sum_{k=0}^{m} \frac{ A^{(n)}(k) \Gamma(k-m)
      }{\Gamma(a_1+c_0+k)\Gamma(a_2+c_0+k)\Gamma(-m)} (1-z)^m\notag\\
    & +  (1-z)^{c_0}\sum_{m=0}^\infty (-1)^{c_0+1} 
      \frac{\Gamma(a_1+c_0+m)\Gamma(a_2+c_0+m)}{\Gamma(1+c_0+m)\Gamma(m+1)}\notag\\ 
     & \cdot
      \sum_{k=0}^{m} \frac{
      A^{(n)}(k)\Gamma(k-m)}{\Gamma(a_1+c_0+k)\Gamma(a_2+c_0+k)\Gamma(-m)}\notag\\ 
     & \cdot
       \left(\psi(a_1+c_0+m)+\psi(a_2+c_0+m)-\psi(1+c_0+m)-\psi(1+m-k)\right)  
      (1-z)^m\notag\\
     &+\sum_{m=0}^{c_0-1}(-1)^m\frac{\Gamma(a_1+m)\Gamma(a_2+m)}{\Gamma(m+1)}\sum_{k=0}^\infty
       \frac{A^{(n)}(k)\Gamma(c_0+k-m)}{\Gamma(a_1+c_0+k)\Gamma(a_2+c_0+k)}(1-z)^m
       \notag\\    
     &+(1-z)^{c_0}\sum_{m=0}^\infty(-1)^{m+c_0}\frac{\Gamma(a_1+c_0+m)\Gamma(a_2+c_0+m)}{\Gamma(1+c_0+m)}
       \notag\\     
     &\cdot \sum_{k=m+1}^\infty
       \frac{A^{(n)}(k)\Gamma(k-m)}{\Gamma(a_1+c_0+k)\Gamma(a_2+c_0+k)}(1-z)^m\notag 
  \end{align*}
  Rewriting this expression in terms of Pochhammer symbols yields the claim.
\end{proof}
For the applications in Section~\ref{sec:applications} we need the case
$n=4$ and $c_0=1$ explicitly.  The four coefficients to be determined
are
\begin{align}
  \label{eq:l0}
  l_0 &= 
    \frac{\Gamma(a_1)\Gamma(a_2)}{\Gamma(1+a_1)\Gamma(1+a_2)}\sum_{k=0}^\infty
    \frac{\Gamma(1+k)}{(a_1+1)_k(a_2+1)_k}A^{(4)}(k)\\
  \label{eq:w0}
  w_0 &= A^{(4)}(0)\\
  \label{eq:q0}
  q_0 &= A^{(4)}(0) \left( \psi(a_1+1) +\psi(a_2+1) -\psi(1) - \psi(2) \right)\\
    &\phantom{=} -\sum_{k=1}^\infty
    \frac{\Gamma(k)}{(a_1+{1})_k(a_2+{1})_k}A^{(4)}(k)\notag\\
  \label{eq:q1}
  q_1 &= \frac{(a_1+{1})(a_2+{1})}{2} \left( \sum_{k=0}^1
    \frac{(-1)_k}{(a_1+{1})_k(a_2+{1})_k}A^{(4)}(k) \right. \\
    &\phantom{=} \cdot \left( \psi(a_1+2) + \psi(a_2+2) - \psi(2-k) - \psi(3) \right) \notag\\
    &\phantom{=} \left. +\sum_{k=2}^\infty
    \frac{\Gamma(k-1)}{(a_1+{1})_k(a_2+{1})_k}A^{(4)}(k) \right) \notag
\end{align}
with
\begin{align}
  A^{(4)}(k) &=\frac{\Gamma(b_{3}+b_{2} - a_4 - a_{3} + k)
    \Gamma(b_1-a_3+k)}{\Gamma(b_{3}+b_{2} - a_4 - a_{3})
    \Gamma(b_1-a_3) \Gamma(1+k)} \\
  &\phantom{=} \cdot \,{}_3F_2
  \left(\genfrac{}{}{0pt}{}{b_3-a_4,b_2-a_4,-k}{b_3+b_2-a_4-a_3,
      1+a_3-b_1-k};1\right)\notag
\end{align}

\begin{rem}
  \label{rem:33}
  The analogous result for $c$ a nonpositive integer can be derived in
  a parallel manner. The evaluation of the $u$ integral now yields a
  different result, see again~\cite {Erdelyi:1953ab}
  or~\cite{Norlund:1963ab}. The final result for the analytic
  continuation of ${}_nF_{n-1}$ in this case can be found
  in~\cite{Buehring:1992ab}. 
\end{rem}

As a corollary to Proposition~\ref{thm:12} we obtain the following
series expansion in powers of $1-z$ of $G_p(z)$. We start with an easy lemma:
\begin{lem}
 \label{lem:11} For any $1 < p \leq q \leq n$, if $|z-1| < 1$, then
  \[
  \begin{aligned}
    &G_p
    \left(\genfrac{}{}{0pt}{0}{\alpha_{1},\dots,\alpha_n}{\gamma_{1},\dots,\gamma_n};z\right) = z^{\gamma_q} \sum_{m=0}^\infty \frac{1}{m!} G_p
    \left(\genfrac{}{}{0pt}{0}{\alpha_{1},\dots,\alpha_n}{\gamma_{1},\dots,\gamma_{q-1},\gamma_q+m,\gamma_{q+1},\dots,\gamma_n};1\right) (1-z)^m
  \end{aligned}
  \]
\end{lem}
\begin{proof}
  This almost immediately follows from the definition of $G_p(z)$ in
  terms of the Meijer G-function in Section~\ref{sec:meijer-g-functions}. In
  fact,
  \begin{align*}
    &G_p
    \left(\genfrac{}{}{0pt}{0}{\alpha_{1},\dots,\alpha_n}{\gamma_{1},\dots,\gamma_n};z\right)\notag \\
    &= z^{\gamma_q} \sum_{m=0}^\infty \frac{1}{m!}
    \left.\frac{\diff{m}{\left(z^{-\gamma_{q}}G_p(z)\right)}}{\diff{}{z^m}}\right|_{z=1}
    (z-1)^m\notag \\
    &= z^{\gamma_q}\sum_{m=0}^\infty \frac{1}{m!} 
      \int \frac{\diff{}{t}}{2\pi i}
      \,\e{i\pi t(p-2)}
    \prod_{j=1}^n
  \frac{\Gamma(\alpha_j+t)}{\Gamma(1-\gamma_j+t)} \prod_{h=1}^p
  \Gamma(\gamma_h-t)\Gamma(1-\gamma_h+t)\notag \\
    &\phantom{=} \phantom{ XX } \left. \cdot \frac{\diff{m}{(z^{t-\gamma_q})}}{\diff{}{z}^m} \right|_{z=1} (z-1)^m\notag \\
    &= z^{\gamma_q}\sum_{m=0}^\infty \frac{1}{m!} 
      \int \frac{\diff{}{t}}{2\pi i}
      \,\e{i\pi t(p-2)}
    \prod_{j=1}^n
  \frac{\Gamma(\alpha_j+t)}{\Gamma(1-\gamma_j+t)} \prod_{h=1}^p
  \Gamma(\gamma_h-t)\Gamma(1-\gamma_h+t)\notag \\
    &\phantom{=} \phantom{ XX } \cdot 
      \frac{\Gamma(\gamma_q-t+m)}{\Gamma(\gamma_q-t)} (1-z)^m\notag
\end{align*}
  For $p=2$ this is due to~\cite{Norlund:1955ab}.
\end{proof}

\begin{thm}
  \label{thm:22}
    For any $2 < p \leq q \leq n$, if $|z-1| < 1$, $\Re \beta_n > \Re\beta_p$,
 $\Re(\alpha_s+\gamma_j)>0$, $j=1\dots,p$, $s=p+1,\dots,n$, $\alpha_{p}+\gamma_p,
    \alpha_s+\gamma_{s+1} \not \in \mZ_{\leq 0}$, $s=2,\dots,p-1$ then
  \[
    \begin{aligned}
    G_p(z) &= \sum_{m=0}^\infty \Gamma(\alpha_1+\gamma_2) \int \frac{\diff{}{v}}{2\pi i} \e{-i\pi v} \Gamma(\alpha_1+\gamma_1+v) \Gamma(-v) \\
   &\phantom{=} \cdot \int \frac{\diff{}{s}}{2 \pi i} \frac{B_{p,m}(s)}{\Gamma(m+1)} \frac{ \Gamma(\gamma_2-s) \Gamma(\gamma_1+v-s) }{\Gamma(\alpha_1+\gamma_1+\gamma_2+v-s)}\\
   &\phantom{=} \cdot \int
    \frac{\diff{}{u}}{2\pi i} \e{-i\pi u} 
    \frac{\Gamma(-v+u)\Gamma(-u)}{\Gamma(-v)} \prod_{s=p+1}^n
    \frac{\Gamma(\alpha_s+\gamma_1+u)}{\Gamma(1-\gamma_s+\gamma_1+u)} (1-z)^m
  \end{aligned}
  \]
  where $B_{p,m}(s) = B_p(s)|_{\gamma_p \to \gamma_p+m}$ with $B_p(s)$
  as in Proposition~\ref{prop:10}. If $p=2$
  then
  \[
  \begin{aligned}
    G_2(z) &= \sum_{m=0}^\infty
    \frac{\Gamma(\alpha_1+\gamma_2+m)\Gamma(\alpha_2+\gamma_2+m)}{\Gamma(m+1)}
    \\
   &\phantom{=} \cdot \int \frac{\diff{}{v}}{2\pi i} \e{-i\pi v} \frac{\Gamma(\alpha_1+\gamma_1+v) \Gamma(\alpha_2+\gamma_1+v)  \Gamma(-v)}{\Gamma(\alpha_1+\alpha_2+\gamma_1+\gamma_2+m+v)} \\
   &\phantom{=} \cdot \int
    \frac{\diff{}{u}}{2\pi i} \e{-i\pi u} 
    \frac{\Gamma(-v+u)\Gamma(-u)}{\Gamma(-v)} \prod_{s=3}^n
    \frac{\Gamma(\alpha_s+\gamma_1+u)}{\Gamma(1-\gamma_s+\gamma_1+u)} (1-z)^m
  \end{aligned}
  \]
\end{thm}
\begin{proof}
  This follows from Proposition~\ref{thm:12} and
  Lemma~\ref{lem:11}. The $s$ integral can be expressed in terms
  multiple integrals of hypergeometric functions. By
  Remark~\ref{rem:24} the $u$ 
  integral is the evaluation at $z=1$ of a hypergeometric function
  which requires a convergence condition that follows from the
  following asymptotic expansion~\cite{Norlund:1955ab}.
  \[
  {}_{n+1} F_{n}
    \left(\genfrac{}{}{0pt}{}{a_1,\dots,a_n,-x}{b_1,\dots,b_n};z\right)
    \sim \sum_{i=1}^n C_i x^{-a_i} (\log x)^{r_i} ,
  \]
  for some constants $C_i$, nonnegative integers $r_i$ and $|z-1| < 1$.
\end{proof}

\begin{rem}
  \label{rem:31}
  By using the residue theorem, we can give an integral representation
  for $G_p(z)$ as a function of $1-z$.
  \[
    \begin{aligned}
    G_p(z) &= \Gamma(\alpha_1+\gamma_2)  \int \frac{\diff{}{t}}{2\pi
      i} \e{i\pi(t+1)} \Gamma(-t) \int \frac{\diff{}{v}}{2\pi i} \e{-i\pi v} \Gamma(\alpha_1+\gamma_1+v) \Gamma(-v) \\
   &\phantom{=} \cdot \int \frac{\diff{}{s}}{2 \pi i} B_{p}(s,t) \frac{ \Gamma(\gamma_2-s) \Gamma(\gamma_1+v-s) }{\Gamma(\alpha_1+\gamma_1+\gamma_2+v-s)}\\
   &\phantom{=} \cdot \int
    \frac{\diff{}{u}}{2\pi i} \e{-i\pi u} 
    \frac{\Gamma(-v+u)\Gamma(-u)}{\Gamma(-v)} \prod_{s=p+1}^n
    \frac{\Gamma(\alpha_s+\gamma_1+u)}{\Gamma(1-\gamma_s+\gamma_1+u)} (1-z)^t
  \end{aligned}
  \]
  where $B_{p}(s,t) = B_p(s)|_{\gamma_p \to \gamma_p+t}$.
\end{rem}

\subsection{Examples}
\label{sec:examples}

For our applications in Section~\ref{sec:applications} we need a few cases explicitly.
\begin{expl}
  \label{expl:27}
  For $n=3, p=2$ we have
  \begin{equation}
    \label{eq:G2_n=3}
    G_2(z) = z^{\gamma_1} \sum_{m=0}^\infty h_m\, (1-z)^m
  \end{equation}
  with
  \[
  \begin{aligned}
    h_m&=
    \frac{\Gamma(\alpha_3+\gamma_1)\Gamma(\alpha_1+\gamma_2+m)\Gamma(\alpha_2+\gamma_2+m)
    }{\Gamma(1-\gamma_3+\gamma_1)\Gamma(m+1)}
    \\
    &\phantom{=} \cdot \int \frac{\diff{}{v}}{2\pi i} \e{-i\pi v}
    \frac{\Gamma(\alpha_1+\gamma_1+v) \Gamma(\alpha_2+\gamma_1+v)
      \Gamma(-v)}{\Gamma(\alpha_1+\alpha_2+\gamma_1+\gamma_2+m+v)}
    {}_{2}F_{1}\left(\genfrac{}{}{0pt}{0}{-v,\alpha_{3}+\gamma_1}{1-\gamma_{3}+\gamma_1};1\right)
  \end{aligned}
  \]
  By applying Gauss' formula~\eqref{eq:Gauss} this becomes
  \begin{align}
    h_m&= 
         \frac{\Gamma(\alpha_3+\gamma_1)\Gamma(\alpha_1+\gamma_2+m)\Gamma(\alpha_2+\gamma_2+m)
         }{\Gamma(1-\alpha_3-\gamma_3)\Gamma(m+1)} \notag\\
       &\phantom{=} \cdot \int \frac{\diff{}{v}}{2\pi i} \e{-i\pi v}
         \frac{\Gamma(\alpha_1+\gamma_1+v) \Gamma(\alpha_2+\gamma_1+v)
         \Gamma(1-\alpha_{3}-\gamma_{3}+v)
         \Gamma(-v)}{\Gamma(\alpha_1+\alpha_2+\gamma_1+\gamma_2+m+v)
         \Gamma(1+\gamma_1-\gamma_{3}+v)}\notag\\
    \label{eq:hm_n=3}
       &=  \frac{\Gamma(\alpha_1+\gamma_1)
         \Gamma(\alpha_2+\gamma_1)\Gamma(\alpha_3+\gamma_1)\Gamma(\alpha_1+\gamma_2+m)\Gamma(\alpha_2+\gamma_2+m)
         }{\Gamma(\alpha_1+\alpha_2+\gamma_1+\gamma_2+m)\Gamma(1+\gamma_1-\gamma_{3})\Gamma(m+1)}\\
       & \phantom{=} \cdot
         {}_{3}F_{2}\left(\genfrac{}{}{0pt}{0}{\alpha_1+\gamma_1,\alpha_2+\gamma_1,1-\alpha_{3}-\gamma_{3}}{\alpha_1+\alpha_2+\gamma_1+\gamma_2+m,1+\gamma_1-\gamma_{3}};1\right)
         \notag
  \end{align}
\end{expl}

  \begin{expl}
  \label{expl:28}
  For $n=4, p=2$ we have
  \begin{equation}
    \label{eq:G2_n=4}
    G_2(z) = z^{\gamma_1} \sum_{m=0}^\infty h_m\, (1-z)^m
  \end{equation}
  with
  \[
  \begin{aligned}
    h_m&= \frac{\Gamma(\alpha_3+\gamma_1) \Gamma(\alpha_4+\gamma_1)
      \Gamma(\alpha_1+\gamma_2+m)\Gamma(\alpha_2+\gamma_2+m)}{\Gamma(1-\gamma_3+\gamma_1)
      \Gamma(1-\gamma_4+\gamma_1) \Gamma(m+1)}  \\
    &\phantom{=} \cdot \int \frac{\diff{}{v}}{2\pi i} \e{-i\pi v} \frac{\Gamma(\alpha_1+\gamma_1+v) \Gamma(\alpha_2+\gamma_1+v)  \Gamma(-v)}{\Gamma(\alpha_1+\alpha_2+\gamma_1+\gamma_2+m+v)} \\
    &\phantom{=} \cdot
    {}_{3}F_{2}\left(\genfrac{}{}{0pt}{0}{-v,\alpha_{3}+\gamma_1,\alpha_4+\gamma_1}{1-\gamma_{3}+\gamma_1,1-\gamma_4+\gamma_1};1\right)
  \end{aligned}
  \]
  We evaluate the $v$ integral by closing the contour to the right and
  obtain
  \begin{align}
    \label{eq:hm_n=4}
    h_m&= \frac{\Gamma(\alpha_3+\gamma_1) \Gamma(\alpha_4+\gamma_1)
         \Gamma(\alpha_1+\gamma_2+m)\Gamma(\alpha_2+\gamma_2+m)}{\Gamma(1-\gamma_3+\gamma_1)
         \Gamma(1-\gamma_4+\gamma_1) \Gamma(m+1)}  \\ 
       &\phantom{=} \cdot \sum_{\ell=0}^\infty \frac{\Gamma(\alpha_1+\gamma_1+\ell) \Gamma(\alpha_2+\gamma_1+\ell) }{\Gamma(\alpha_1+\alpha_2+\gamma_1+\gamma_2+m+\ell)\Gamma(\ell+1)} 
         {}_{3}F_{2}\left(\genfrac{}{}{0pt}{0}{-\ell,\alpha_{3}+\gamma_1,\alpha_4+\gamma_1}{1-\gamma_{3}+\gamma_1,1-\gamma_4+\gamma_1};1\right)\notag
  \end{align}
  \end{expl}

For the last example we will need a theorem of Slater which we state here
as lemma
\begin{lem}
  \label{lem:26}
  Let
  $p,q,r,s \in \mZ_{\geq 0}$, $a_1,\dots,a_p$,
  $b_1,\dots,b_q,c_1,\dots,c_r,d_1,\dots,d_s \in \mC$ and
  \[
  I(z) = \int \frac{\diff{}{t}}{2\pi i} \frac{\prod_{i=1}^p
    \Gamma(a_i+t) \prod_{j=1}^q \Gamma(b_j-t)}{\prod_{k=1}^r
    \Gamma(c_k+t) \prod_{\ell=1}^s
    \Gamma(d_\ell-t)}z^t\\
  \]
  Then
  \[
  \begin{aligned}
    I(z) &= \sum_{m=1}^q z^{b_m} \frac{\prod_{i=1}^p \Gamma(a_i+b_m)
      \prod_{j=1,j\not= m}^q \Gamma(b_j-b_m)}{\prod_{k=1}^r
      \Gamma(c_k+b_m)
      \prod_{\ell=1}^s \Gamma(d_\ell-b_m)} \\
    &\phantom{=} \cdot
    \,{}_{p+s}F_{q+r-1}\left(\genfrac{}{}{0pt}{0}{a_1+b_m,\dots,a_p+b_m,1+b_m-d_1,\dots,1+b_m-d_s}{c_1+b_m,\dots,c_r+b_m,1+b_m-b_1,\widehat{\dots},1+b_m-b_q};(-1)^{q+s}z\right)\\
  \end{aligned}
  \]
  provided that $\lambda-\mu \not \in \mZ$ for all
  $\lambda,\mu \in \{b_1,\dots,b_q,c_1\dots,c_r\}$ and
  \[
  \tfrac{1}{2}\pi | p+q-r-s| > |\arg z |, \quad q+r \geq p+s, \quad
  |z| < 1,
  \]
  Also,
  \[
  \begin{aligned}
    I(z) &= \sum_{m=1}^p z^{-a_m} \frac{\prod_{i=1,i\not= m}^p
      \Gamma(a_i-a_m) \prod_{j=1}^q \Gamma(b_j+a_m)}{\prod_{k=1}^r
      \Gamma(c_k-a_m)
      \prod_{\ell=1}^s \Gamma(d_\ell+a_m)} \\
    &\phantom{=} \cdot
    \,{}_{q+r}F_{p+s-1}\left(\genfrac{}{}{0pt}{0}{b_1+a_m,\dots,b_q+a_m,1+a_m-c_1,\dots,1+a_m-c_r}{d_1+a_m,\dots,d_s+a_m,1+a_m-a_1,\widehat{\dots},1+a_m-a_p};\frac{(-1)^{p+r}}{z}\right)\\
  \end{aligned}
  \]
  provided that $\lambda-\mu \not \in \mZ$ for all
  $\lambda,\mu \in \{a_1,\dots,a_p,d_1\dots,d_s\}$ and
  \[
  \tfrac{1}{2}\pi | p+q-r-s| > |\arg z |, \quad p+s \geq q+r, \quad
  |z| < 1.
  \]
  In addition, both formulas are valid for $z = 1$ if furthermore
  \[
  \Re\left(\sum_{k=1}^r c_k + \sum_{\ell=1}^s d_k - \sum_{i=1}^p a_i
    - \sum_{j=1}^q b_j \right) > 0\,.
  \]  
\end{lem}
\begin{proof}
  See~\cite{Slater:1966ab}. In fact, Lemma~\ref{lem:4}(1) is special
  case of this lemma.
\end{proof}

  \begin{expl}
  \label{expl:29}
  Finally, for $n=4, p=3$ we have
  \begin{equation}
    \label{eq:G3_n=4}
    G_3(z) = z^{\gamma_1} \sum_{m=0}^\infty k_m\, (1-z)^m
  \end{equation}
  with
  \begin{align}
    \label{eq:km-int}
    k_m&= \frac{\Gamma(\alpha_1+\gamma_2) \Gamma(\alpha_2+\gamma_3+m)
         \Gamma(\alpha_3+\gamma_3+m)
         \Gamma(\alpha_4+\gamma_1)}{\Gamma(1-\alpha_4-\gamma_4)
         \Gamma(m+1)} \\ 
       &\phantom{=} \cdot \int 
         \frac{\diff{}{v}}{2\pi i} \e{-i\pi v} 
         \frac{\Gamma(\alpha_1+\gamma_1+v)\Gamma(1-\alpha_4-\gamma_4+v) \Gamma(-v)}{\Gamma(1-\gamma_4+\gamma_1+v)}\notag\\ 
       &\phantom{=} \cdot \int \frac{\diff{}{s}}{2 \pi i}
         \e{-i\pi s}
         \frac{\Gamma(\alpha_2+s)\Gamma(\alpha_3+s)  \Gamma(\gamma_2-s)
         \Gamma(\gamma_1+v-s) }{\Gamma(\alpha_2+\alpha_3+\gamma_3+m+s)
         \Gamma(\alpha_1+\gamma_1+\gamma_2+v-s)}\notag
  \end{align}
  where we have used~\eqref{eq:B3} and Gauss'
  formula~\eqref{eq:Gauss}. The two integrals can be evaluated with
  the help of Lemma~\ref{lem:26}. In particular, if
  $\Re(\alpha_1+\gamma_3+m) > 0$ and if
  $\alpha_1,\dots,\alpha_3,\gamma_1,\dots,\gamma_3$ are such that
  there are only simple poles, we obtain for the $s$ integral in
  \[
  \begin{aligned}
    &\int \frac{\diff{}{s}}{2 \pi i}\e{- \pi i
      s}\frac{\Gamma(\alpha_2+s)\Gamma(\alpha_3+s) \Gamma(\gamma_1-s)
      \Gamma(\gamma_2+v-s)}{\Gamma(\alpha_2+\alpha_3+\gamma_3+m+s)
      \Gamma(\alpha_1+\gamma_1+\gamma_2+v-s)}\\
    &= \e{-i\pi\alpha_2}
    \frac{\Gamma(\alpha_3-\alpha_2)\Gamma(\gamma_1+\alpha_2)\Gamma(\gamma_2+v+\alpha_2)}{\Gamma(\alpha_3+\gamma_3+m)\Gamma(\alpha_1+\alpha_2+\gamma_1+\gamma_2+v)}\\
    &\phantom{=} \cdot
    {}_3F_2\left(\genfrac{}{}{0pt}{0}{\gamma_1+\alpha_2,\gamma_2+v+\alpha_2,1-\alpha_3-\gamma_3-m}{1+\alpha_2-\alpha_3,\alpha_1+\alpha_2+\gamma_1+\gamma_2+v};1\right)\\
    &\phantom{=} + \e{-i\pi\alpha_3}
    \frac{\Gamma(\alpha_2-\alpha_3)\Gamma(\gamma_1+\alpha_3)\Gamma(\gamma_2+v+\alpha_3)}{\Gamma(\alpha_2+\gamma_3+m)\Gamma(\alpha_1+\alpha_3+\gamma_1+\gamma_2+v)}\\
    &\phantom{=} \cdot
    {}_3F_2\left(\genfrac{}{}{0pt}{0}{\gamma_1+\alpha_3,\gamma_2+v+\alpha_3,1-\alpha_2-\gamma_3-m}{1-\alpha_2+\alpha_3,\alpha_1+\alpha_3+\gamma_1+\gamma_2+v};1\right)\\
  \end{aligned}
  \]
  Evaluting the $v$ integral in~\eqref{eq:km-int} with the residue
  theorem finally yields
  \begin{align}
    \label{eq:km}
    k_m&= \frac{\Gamma(\alpha_1+\gamma_2) \Gamma(\alpha_2+\gamma_3+m)
         \Gamma(\alpha_3+\gamma_3+m)
         \Gamma(\alpha_4+\gamma_1)}{\Gamma(1-\alpha_4-\gamma_4)
         \Gamma(m+1)} \\ 
       &\phantom{=} \cdot \sum_{\ell=0}^\infty 
         \frac{\Gamma(\alpha_1+\gamma_1+\ell)\Gamma(1-\alpha_4-\gamma_4+\ell) }{\Gamma(1-\gamma_4+\gamma_1+\ell) \Gamma(\ell+1)}\notag\\ 
       &\phantom{=} \cdot \left( 
         \e{-i\pi\alpha_2}
         \frac{\Gamma(\alpha_3-\alpha_2)\Gamma(\gamma_1+\alpha_2)\Gamma(\gamma_2+\ell+\alpha_2)}{\Gamma(\alpha_3+\gamma_3+m)\Gamma(\alpha_1+\alpha_2+\gamma_1+\gamma_2+\ell)}
         \right. \notag\\
       &\phantom{=} \cdot
         {}_3F_2\left(\genfrac{}{}{0pt}{0}{\gamma_1+\alpha_2,\gamma_2+\ell+\alpha_2,1-\alpha_3-\gamma_3-m}{1+\alpha_2-\alpha_3,\alpha_1+\alpha_2+\gamma_1+\gamma_2+\ell};1\right)
         \notag\\ 
       &\phantom{=} + \e{-i\pi\alpha_3}
         \frac{\Gamma(\alpha_2-\alpha_3)\Gamma(\gamma_1+\alpha_3)\Gamma(\gamma_2+\ell+\alpha_3)}{\Gamma(\alpha_2+\gamma_3+m)\Gamma(\alpha_1+\alpha_3+\gamma_1+\gamma_2+\ell)}\notag\\
       &\phantom{=} \cdot \left. 
         {}_3F_2\left(\genfrac{}{}{0pt}{0}{\gamma_1+\alpha_3,\gamma_2+\ell+\alpha_3,1-\alpha_2-\gamma_3-m}{1-\alpha_2+\alpha_3,\alpha_1+\alpha_3+\gamma_1+\gamma_2+\ell};1\right) \right)
         \notag
  \end{align}
  \end{expl}

\section{Applications}
\label{sec:applications}

\subsection{The Frobenius method}
\label{sec:frobenius-method}

The linear differential equation~\eqref{eq:DE} can equivalently be
written as a first order matrix differential equation 

\begin{equation}
  \label{eq:DE1}
  \theta\, Y(z) = A(z) Y(z) 
\end{equation}
where
\[
\begin{aligned}
  A(z) &=
  \begin{pmatrix}
    0 & 1 & 0 & \cdots & 0\\
    0 & 0 & 1 & \cdots & 0\\
    \vdots & \vdots & \vdots & \ddots & \vdots\\
    0 & 0 & 0 & \cdots & 1\\
    a_1 & a_2 & a_3 & \cdots & a_n
  \end{pmatrix}, 
  &
  Y(z) &=
  \begin{pmatrix}
    y(z)\\
    \theta\,y(z) \\
    \vdots\\
    \theta^{n-2}\,y(z) \\
    \theta^{n-1}\,y(z) \\    
  \end{pmatrix}
\end{aligned}
\]
where 
\[a_i = \frac{e_{n+1-i}(\gamma_1,\dots,\gamma_{n-1}) - z\,
e_{n+1-i}(\alpha_1,\dots,\alpha_n)}{1-z},\]
$e_i(x_1,\dots,x_k)$ is the
elementary symmetric polynomial of degree $i$, and $y(z)$ is a
solution to~\eqref{eq:DE}. For linearly independent solutions
$y_1,\dots,y_n$ to~\eqref{eq:DE} we have linearly independent
solution vectors $Y_1, \dots, Y_n$ to~\eqref{eq:DE1}. We collect
these column vectors into a matrix $\Phi = \left(
  Y_1 \dots Y_n\right)$, known as a fundamental matrix, since it
satisfies $\theta\, \Phi = A \Phi$. Near any regular singularity
$z=z_0$ of~\eqref{eq:DE1}, there exist a constant $n\times n$ matrix
$R$, a real number $r>0$, and an $n\times n$ matrix $S$ of
singly-valued holomorphic functions in the annulus $0 < |z-z_0| < r$
such that the fundamental matrix takes the following
form~\cite{Coddington:1955ab}  
\[
  \Phi(z) = S(z) (z-z_0)^R.
\]
The matrix $R$ is determined by $A$, e.g. if the eigenvalues of $A(0)$
do not differ by positive integers, then $R=A(0)$. The matrix $S(z)$
can be determined as follows. For the first row of $S(z)$ we make a
power series ansatz and substitute it into~\eqref{eq:DE1}. This yields
recursion relations for the coefficients of the power series. These
recursion relations can be solved after choosing a number of
constants. This is known as the Frobenius method. The remaining rows
of $S(z)$ are obtained from the first row by successively acting with
$\theta$ on the first row.

Note that the fundamental matrix is not unique. Multiplication by any
invertible constant $n\times n$ matrix $C$ yields another fundamental
matrix. We will use this freedom in the examples to choose bases
which are easy to relate with the basis elements $y_j^*(z), y_{ij}(z),
G_p(z)$, or $\xi_n(z)$ given in terms of integral representations.

\subsection{The mirror quartic}
\label{sec:mirror-quartic}

The variation of polarized Hodge structure of the family $\pi: \cX \to
\mP^1$ of mirror quartics given as
\[
  \cX_z = \{{x_0}^4 + {x_1}^4 + {x_2}^4 + {x_3}^4 -
  4\,z^{-\frac{1}{4}}x_0x_1x_2x_3 = 0\} \subset \mP^3 
\]
leads to a Picard--Fuchs equation which is a hypergeometric
differential equation of order 3 with exponents~\cite{Nagura:1995kd}
\[
\begin{aligned}
  \alpha_i &= \frac{i}{4}, & \gamma_i &= 0, & i& =1,2,3, & \beta_3 &= \frac{1}{2}.
\end{aligned}
\]


A fundamental system at $z=0$ is $\Phi_0(z) = S_0(z) z^{R_0} 4^{-4R_0}
C_0$ with
\[
\begin{aligned}
  R_0 &=
  \begin{pmatrix}
    0 & 1 & 0 \\
    0 & 0 & 1 \\
    0 & 0 & 0
  \end{pmatrix},
  &
  C_0 &=
  \begin{pmatrix}
    1 & 0 & \frac{1}{4} \\
    0 & \frac{1}{2\pi i} & 0 \\
    0 & 0 & \frac{1}{(2\pi i)^2} 
  \end{pmatrix}
\end{aligned}
\]
and
\[
  \begin{aligned}
    S_{0,11} &= 1+{\tfrac {3}{32}}\,z+{\tfrac {315}{8192}}\,{z}^{2}+{\tfrac {5775}{262144
}}\,{z}^{3} +O(z^4) \\
    S_{0,12} &=  {\tfrac {13}{32}}\,z+{\tfrac {3069}{16384}}\,{z}^{2}+{\tfrac {176005}{
1572864}}\,{z}^{3} + O(z^4) \\
    S_{0,13} &=  {\tfrac {169}{2048}}\,{z}^{2}+{\tfrac {35841}{524288}}\,{z}^{3} + O(z^4)\\
  \end{aligned}
\]
A fundamental system at $y=1-z$ is $\Phi_1(y) = S_1(y) y^{R_1}$ with
\[
  R_1 =
  \begin{pmatrix}
    0 & 0 & 0 \\
    0 & 1 & 0 \\
    0 & 0 & \frac{1}{2} 
  \end{pmatrix}
\]
and
\[
  \begin{aligned}
    S_{1,11} &= 1-\tfrac{1}{32}\,{y}^{2}-{\tfrac {131}{3840}}\,{y}^{3}+O(y^4)\\
    S_{1,12} &= 1+{\tfrac {35}{48}}\,y+{\tfrac {665}{1152}}\,{y}^{2} + O(y^3)\\
    S_{1,13} &= 1+{\tfrac {11}{24}}\,y+{\tfrac
        {39}{128}}\,{y}^{2}+{\tfrac {1181}{5120}} \,{y}^{3} + O(y^4)\\
  \end{aligned}
\]
\begin{thm}
  \label{thm:25}
  The analytic continuation of $\Phi_0(z)$ to $z=1$ is determined by
  $\Phi_0(z) = \Phi_1(1-z) M_{10}$ with
  \[
    M_{10} =
    \begin{pmatrix}
      \frac{A}{2\,\sqrt{2}\pi} & -\frac{A}{4\pi i}  & 0\\
      \frac{2}{\sqrt{2}\pi}
      \left(\frac{3\,A}{64} + \frac{1}{A}\right) &  -\frac{1}{\pi i}
      \left(\frac{3\,A}{64} - \frac{1}{A}\right) & 0 \\
      -\frac{2}{\sqrt{2}\pi} & 0 & -\frac{1}{\sqrt{2}\pi}
    \end{pmatrix}
  \]
  where $A=\frac{\Gamma(\frac{1}{8})
    \Gamma(\frac{3}{8})}{\Gamma(\frac{5}{8}) \Gamma(\frac{7}{8})}$. 
\end{thm}
\begin{proof}
  By evaluating the residues in the definition of $G_p(z)$ closing the
  contour to the right one finds that 
\[
  \begin{aligned}
    G_1(z) &= \Gamma(\tfrac{1}{4}) \Gamma(\tfrac{1}{2}) \Gamma(\tfrac{3}{4})\Phi_{0,11},\\
   \tfrac{1}{2\pi i} G_2(z) &= -\Gamma(\tfrac{1}{4}) \Gamma(\tfrac{1}{2}) \Gamma(\tfrac{3}{4})\Phi_{0,12},\\
    \tfrac{1}{(2\pi i)^2} G_3(z) &= \Gamma(\tfrac{1}{4}) \Gamma(\tfrac{1}{2}) \Gamma(\tfrac{3}{4}) \left(\Phi_{0,13} +
      \tfrac{1}{2}\Phi_{0,12} - \tfrac{1}{2} \Phi_{0,11} \right)  .\\
  \end{aligned}
\]
Corollary~\ref{cor:20} now yields
\[
\Gamma(\tfrac{1}{4}) \Gamma(\tfrac{1}{2}) \Gamma(\tfrac{3}{4})
\Phi_{0,11} = g_0(0) \Phi_{1,11} + g_1(0) \Phi_{1,12} +
g_0(\tfrac{1}{2}) \Phi_{1,13}
\]
where the expansion coefficients are determined
by~\eqref{eq:g0_c},~\eqref{eq:g0_0} and~\eqref{eq:g1_0} as 
\[
  \begin{aligned}
    g_0(0) &= \frac{\Gamma(\frac{1}{4})
      \Gamma(\frac{1}{2})^2}{\Gamma(\frac{3}{4})}
    \Hyp{\frac{1}{4},\frac{1}{2},\frac{1}{4}}{\frac{3}{4},1}{1}
    = \frac{\Gamma(\frac{1}{2})\Gamma(\frac{1}{8}) \Gamma(\frac{3}{8})}{2\,\Gamma(\frac{5}{8}) \Gamma(\frac{7}{8})}\\
    g_1(0) &=
    -\frac{\Gamma(\frac{5}{4})\Gamma(\frac{3}{2})\Gamma(-\frac{1}{2})}{\Gamma(
      \frac{3}{4})}
    \Hyp{-\frac{1}{2},\frac{1}{4},\frac{1}{4}}{\frac{3}{4},1}{1}\\
    &= 2\,\Gamma(\tfrac{1}{2})
    \left( \frac{3\,\Gamma(\tfrac{1}{8})\Gamma(\frac{3}{8})}{64\, \Gamma(\frac{5}{8}) \Gamma(\tfrac{7}{8})} + \frac{\Gamma(\tfrac{5}{8})\Gamma(\frac{7}{8})}{\Gamma(\frac{1}{8}) \Gamma(\tfrac{3}{8})}\right)\\
    g_0(\tfrac{1}{2}) &= -2\,\Gamma(\tfrac{1}{2})  
  \end{aligned}
\]
from which the first row of $M_{01}$ follows using $\Gamma(\frac{1}{4})\Gamma(\frac{3}{4})=\sqrt{2}\pi$.

Theorem~\ref{thm:22} (see Example~\ref{expl:27}) yields
\[
- 2 \pi i \Gamma(\tfrac{1}{4}) \Gamma(\tfrac{1}{2})
\Gamma(\tfrac{3}{4}) \Phi_{0,12} = h_0 \Phi_{1,11} + h_1 \Phi_{1,12} 
\]
where the expansion coefficients are determined by~\eqref{eq:hm_n=3} as
\[
  \begin{aligned}
    h_0 &=\Gamma(\tfrac{1}{4})^2\Gamma(\tfrac{1}{2})^2
    \Hyp{\frac{1}{4},\frac{1}{2},\frac{1}{4}}{\frac{3}{4},1}{1} =
    \frac{\Gamma(\frac{1}{4})\Gamma(\frac{1}{2})\Gamma(\frac{3}{4})\Gamma(\frac{1}{8})\Gamma(\frac{3}{8})}{2\,\Gamma(\frac{5}{8})\Gamma(\frac{7}{8})}\\
    h_1 &= \tfrac{1}{6}
    \Gamma(\tfrac{1}{4})^2\Gamma(\tfrac{1}{2})^2\Hyp{\frac{1}{4},\frac{1}{2},\frac{1}{4}}{\frac{7}{4},1}{1}\\
    &= 2\,\Gamma(\tfrac{1}{4})\Gamma(\tfrac{1}{2}) \Gamma(\tfrac{3}{4})
    \left(
      \frac{3\,\Gamma(\frac{1}{8})\Gamma(\frac{3}{8})}{64\,\Gamma(\frac{5}{8})\Gamma(\frac{7}{8})}
      -
      \frac{\Gamma(\frac{5}{8})\Gamma(\frac{7}{8})}{\Gamma(\frac{1}{8})\Gamma(\frac{3}{8})}
    \right)
  \end{aligned}
\]
from which the second row of $M_{01}$ follows.

It remains to explain the evaluation of the various ${}_3F_2$ at
$1$. The basic formula is Dixon's identity~\cite{Slater:1966ab}:
\begin{multline*}
  \,{}_3 F_{2}\left(\genfrac{}{}{0pt}{}{a_1,a_2,
      a_3}{1+a_1-a_2,1+a_1-a_3};1 \right) \\=
  \frac{\Gamma(1+\frac{a_1}{2})\Gamma(1+\frac{a_1}{2}-a_2-a_3)\Gamma(1+a_1-a_2)\Gamma(1+a_1-a_3)}{\Gamma(1+a_1)\Gamma(1+a_1-a_2-a_3)\Gamma(1+\frac{a_1}{2}-a_2)\Gamma(1+\frac{a_1}{2}-a_3)}
\end{multline*}
This identity can be used to evaluate $g_0(0)$ and $h_0$. There is a
generalization of this identity due to Lavoie et
al.~\cite{Lavoie:1994ab}. The two cases we need for $g_1(0)$ and
$h_1$ are
\[
\begin{gathered}
  \Hyp{a_1,a_2,a_3}{a_1-a_2,1+a_1-a_3}{1} = \frac
  {{2}^{-2\,a_{{3}}}\Gamma ( a_{{1}}-a_{{2}} ) \Gamma (
    a_{{1}}-a_{{3}}+1 ) }{\Gamma ( a_{{1}}-2\,a
    _{{3}}+1 ) \Gamma  ( a_{{1}}-a_{{2}}-a_{{3}}+1 ) }\\
  \cdot \left( {\frac {\Gamma ( \frac{a_1}{2}-a_{{3}}+\frac{1}{2} )
        \Gamma ( \frac{a_1}{2}-a_{{2}}-a_{{3}}+1 ) }{\Gamma (
        \frac{a_1}{2}+\frac{1}{2} ) \Gamma ( \frac{a_1}{2}-a_{{2}} )
      }}+{\frac {\Gamma ( \frac{a_1}{2}-a_{{3}}+1 ) \Gamma (
        \frac{a_1}{2}-a_{{2}}-a_{{3}}+\frac{1}{2} ) }{\Gamma (
        \frac{a_1}{2} ) \Gamma (
        \frac{a_1}{2}-a_{{2}}+\frac{1}{2} ) }} \right), \\
  \Hyp{a_1,a_2,a_3}{2+a_1-a_2,1+a_1-a_3}{1} = \frac
  {{2}^{1-2\,a_{{2}}}\Gamma ( a_{{1}}-a_{{3}}+1 ) \Gamma
    ( a_{{1}}-a_{{2}}+2 ) \Gamma ( a_{{2}}-1 )
  }{\Gamma ( a_{{1}}-2\,a_{{2}}+2 ) \Gamma (
      a_{{1}}-a_{{2}}-a_{{3}}+2 ) \Gamma ( a_{{2}} )
  } \\
  \cdot \left( -{\frac {\Gamma ( \frac{a_1}{2}-a_{{2}}+\frac{3}{2} )
        \Gamma ( \frac{a_1}{2}-a_{{3}}-a_{{2}}+2 ) }{\Gamma
        ( \frac{a_1}{2}+\frac{1}{2} ) \Gamma (
          \frac{a_1}{2}-a_{{3}}+1 ) }}+{\frac {\Gamma (
          \frac{a_1}{2}-a_{{2}}+1 ) \Gamma (
          \frac{a_1}{2}-a_{{3}}-a_{{2}}+\frac{3}{2} ) }{\Gamma (
          \frac{a_1}{2} ) \Gamma ( \frac{a_1}{2}-a_{{3}}+\frac{1}{2}
        ) }} \right),
\end{gathered}
\]
respectively. 

Finally, we have $\Phi_{1,13}(y) = \xi_3(z)$. Proposition~\ref{prop:16} with
$ \psi(x) = i\prod_{\nu=1}^3 \left(x-\omega^{-\nu}\right)$, $\omega^4=1$
yields $\psi(1) = 4i, \psi'(1) = 6i,  \frac{\psi''(1)}{2!} = 4i$ so
that 
\[\Phi_{1,13}(y) = -\frac{\Gamma(\frac{3}{2})}{\pi} \left( 2\,y^*_1(z) -3\, y^*_2(z)
      +2\,y^*_3(z) \right).\] 
Lemma~\ref{lem:4}(2) then yields $\Phi_{0,13}(z) = -\frac{1}{\Gamma(\frac{1}{4})\Gamma(\frac{3}{4})}\Phi_{1,13}(y)$. 
\end{proof}

\subsection{The mirror quintic}
\label{sec:mirror-quintic}

The variation of polarized Hodge structure of the family $\pi:\cX \to
\mP^1$ of mirror quintics given as
\[
  \cX_z = \{{x_0}^5 + {x_1}^5 + {x_2}^5 + {x_3}^5 + {x_4}^5 -
  5\, z^{-\frac{1}{5}}x_0x_1x_2x_3x_4 = 0\} \subset \mP^4 
\]
leads to a Picard--Fuchs equation which is a hypergeometric
differential equation of order 4 with exponents~\cite{Candelas:1991rm} 
\[
\begin{aligned}
  \alpha_i &= \frac{i}{5}, & \gamma_i &= 0, & i& =1,2,3,4, & \beta_4 &= 1.
\end{aligned}
\]
A fundamental system at $z=0$ is $\Phi_0(z) = S_0(z)z^{R_0}5^{-5R_0}C_0$ with
\[
\begin{aligned}
  R_0 &=
  \begin{pmatrix}
    0 & 1 & 0 & 0\\
    0 & 0 & 1 & 0\\
    0 & 0 & 0 & 1\\
    0 & 0 & 0 & 0
  \end{pmatrix},
  &
  C_0 &=
  \begin{pmatrix}
    1 & 0 & 0 & 0\\
    0 & \frac{1}{2\pi i} & 0 & 0\\
    0 & 0 &\frac{1}{(2\pi i)^2} &  0\\
    0 & 0 & 0 & \frac{1}{(2\pi i)^3} \\
  \end{pmatrix}
  \cdot
  \begin{pmatrix}
    1 & 0 & -\frac{25}{12} & \frac{200}{(2\pi i)^3} \zeta(3)\\
    0 & 1 & \frac{5}{2} & -\frac{25}{12}\\
    0 & 0 & 5 & 0\\
    0 & 0 & 0 & -5
  \end{pmatrix}
\end{aligned}
\]
and
\[
  \begin{aligned}
    S_{0,11} &=  1 + \tfrac{24}{625}\, z + \tfrac{4536}{390625}\,z^2 + O(z^3)\\
    S_{0,12} &=  \tfrac{154}{625}\, z + \tfrac{32409}{390625}\, z^2 + O(z^3)\\
    S_{0,13} &=  \tfrac{23}{125}\, z + \tfrac{168327}{1562500}\,z^2 + O(z^3)\\
    S_{0,14} &= -\tfrac{46}{125}\,z - \tfrac{26387}{312500}\,z^2 + O(z^3)
  \end{aligned}
\]
The choice of $C_0$ follows from~\cite{Hosono:2000eb}.
A fundamental system at $y=1-z$ is $\Phi_1(y) = S_1(y)y^{R_1}C_1$ with
\[
\begin{aligned}
  R_1 &=
  \begin{pmatrix}
    0 & 0 & 0 & 0 \\
    0 & 1 & 1 & 0\\
    0 & 0 & 1 & 0\\
    0 & 0 & 0 & 2\\
  \end{pmatrix},
  &
  C_1 &= \frac{\sqrt{5}}{4\pi^2} 
  \begin{pmatrix}
    1 & 0 & 0 & 0 \\
    0 & 1 & 0 & 0\\
    0 & 0 & 1 & 0\\
    0 & 0 & 0 & 1\\
  \end{pmatrix}
\end{aligned}
\]
and
\[
  \begin{aligned}
    S_{1,11} &= 1+{\tfrac {2}{625}}{y}^{3}+{\tfrac {97}{18750}}{y}^{4}+{\tfrac {2971}{
468750}}{y}^{5} + O(y^6)\\
    S_{1,12} &= 1+{\tfrac {7}{10}}y+{\tfrac {41}{75}}{y}^{2}+{\tfrac {1133}{2500}}{y}^{3
}+{\tfrac {6089}{15625}}{y}^{4}+{\tfrac {160979}{468750}}{y}^{5} + O(y^6)\\
    S_{1,13} &= -{\tfrac {23}{360}}{y}^{2}-{\tfrac {6397}{60000}}{y}^{3}-{\tfrac {333323
}{2500000}}{y}^{4}-{\tfrac {33777511}{225000000}}{y}^{5}+ O(y^6)\\
    S_{1,14} &= 1+{\tfrac {37}{30}}y+{\tfrac {2309}{1800}}{y}^{2}+{\tfrac {286471}{
225000}}{y}^{3}+{\tfrac {41932661}{33750000}}{y}^{4}+{\tfrac {237108737}
{196875000}}{y}^{5}+ O(y^6)\\
  \end{aligned}
\]
\begin{thm}
  \label{thm:30}
  Let $\Phi_{z_0}(z)$ be the fundamental matrices near $z_0=0,1$.
  The analytic continuation of $\Phi_0(z)$ to $z=1$ is determined by
  $\Phi_0(z) = \Phi_1(1-z) M_{10}$ with
  \[
    M_{10} =
    \begin{pmatrix}
     \medskip
      l_0 & -\frac{h_0}{2\pi i} & \frac{5\,k_0}{(2\pi i)^2} & 0\\
     \medskip
     w_0 & -\frac{h_1}{2\pi i} &\frac{5\,k_1}{(2\pi i)^2} & 2\pi i \\
     \medskip
     1 & 0 & 0 & 0 \\
     w_1 - \frac{7}{10}w_0 & -\frac{h_2}{2\pi i} + \frac{7}{10}\frac{h_1}{2\pi i}
     &\frac{5\,k_2}{(2\pi i)^2} - \frac{7}{10}\frac{5\,k_1}{(2\pi
       i)^2} & 0
    \end{pmatrix}
  \]
  where the real constants $l_0,w_0,w_1,h_0,h_1,h_2$ and the complex
  constants $k_0,k_1,k_2$ are explicitly given in the proof.
\end{thm}

\begin{proof}
  By evaluating the residues in the definition of $G_p(z)$ closing the
  contour to the right one finds that 
\[
\begin{aligned}
    G_1(z) &= \tfrac{4\pi^2}{\sqrt{5}}\Phi_{0,11},\\
   \tfrac{1}{2\pi i} G_2(z) &= -\tfrac{4\pi^2}{\sqrt{5}} \Phi_{0,12},\\
    \tfrac{1}{(2\pi i)^2} G_3(z) &=\tfrac{4\pi^2}{\sqrt{5}}\Phi_{0,13},\\  
    \tfrac{1}{(2\pi i)^3} G_4(z) &= \tfrac{4\pi^2}{\sqrt{5}} \left(
      \tfrac{1}{5} \Phi_{0,14}-\tfrac{1}{5}
  \Phi_{0,13} + \Phi_{0,12}  \right).\\  
\end{aligned}
\]
where the prefactor comes from
$\Gamma(\tfrac{1}{5}) \Gamma(\tfrac{2}{5}) \Gamma(\tfrac{3}{5})
\Gamma(\tfrac{4}{5}) = \frac{4\pi^2}{\sqrt{5}}$.

Corollary~\ref{cor:21} yields
\[
\frac{ 4\pi^2}{\sqrt{5}} \Phi_{0,11} = l_0 + w_0
(1-z)+ q_0 (1-z)\log(1-z)+ w_1 (1-z)^2 + O\left((1-z)^3\right)
\]
and since
$\Phi_{1,12} = (1-z) + \frac{7}{10} (1-z)^2 + O\left((1-z)^3\right)$
this expression can be written as
\[
  \Phi_{0,11} = l_0 \Phi_{1,11} + w_0\Phi_{1,12} + q_0
  \Phi_{1,13} + \left(w_1-\tfrac{7}{10}w_0\right) \Phi_{1,14}
\]
from which the first row of $M_{01}$ follows. The explicit values of
the coefficients are obtained as follows. From~\eqref{eq:A4} we get
\[
A^{(4)}(k) =\frac{(\frac{3}{5})_k (\frac{2}{5})_k}{\Gamma(1+k)} \,{}_3F_2
\left(\genfrac{}{}{0pt}{}{\frac{1}{5},\frac{1}{5},-k}{\frac{3}{5},
    \frac{3}{5}-k};1\right) 
\]
Then,~\eqref{eq:l0},~\eqref{eq:w0},~\eqref{eq:q0} and~\eqref{eq:q1} yield
\begin{align}
  \label{eq:lqw}
  l_0 &= \frac{\Gamma(\frac{1}{5})}{\Gamma(\frac{3}{5})} \sum_{k=0}^\infty
  \frac{\Gamma(\frac{3}{5}+k)\Gamma(\frac{2}{5}+k)}{\Gamma(\frac{6}{5}+k)\Gamma(\frac{7}{5}+k)}
  \,{}_3F_2\left(\genfrac{}{}{0pt}{}{\frac{1}{5},\frac{1}{5},-k}{\frac{3}{5},\frac{3}{5}-k};1\right)\\
    q_0 &= 1,\notag\\
    w_0 & =- \psi(1) - \psi(2) + \psi(\tfrac{6}{5}) + \psi(\tfrac{7}{5})
    -\sum_{k=1}^\infty \frac{(\frac{3}{5})_k
      (\frac{2}{5})_k}{k\,(\frac{6}{5})_k(\frac{7}{5})_k}\,{}_3F_2
    \left(\genfrac{}{}{0pt}{}{\frac{1}{5},\frac{1}{5},-k}{\frac{3}{5},
        \frac{3}{5}-k};1\right)\notag\\
    w_1 &= -\tfrac{21}{25}\left( \vphantom{\sum_{k=2}^\infty} \psi(2) +
      \psi(3) - \psi(\tfrac{11}{5}) - \psi(\tfrac{12}{5}) \right.\notag\\
    &\phantom{=} \left. 
      -\tfrac{1}{6} \left( \psi(1) + \psi(3) - \psi(\tfrac{11}{5}) -
        \psi(\tfrac{12}{5}) \right) \right.\notag\\
    &\phantom{=} \left. - \sum_{k=2}^\infty \frac{(\frac{3}{5})_k
        (\frac{2}{5})_k}{k(k-1)(\frac{6}{5})_k(\frac{7}{5})_k}
      \,{}_3F_2
      \left(\genfrac{}{}{0pt}{}{\frac{1}{5},\frac{1}{5},-k}{\frac{3}{5},
          \frac{3}{5}-k};1\right) \right)\notag
\end{align}

Next, Theorem~\ref{thm:22} (see Example~\ref{expl:28}) yields
\[
G_2(z) = h_0 + h_1 (1-z) + h_2 (1-z)^2 + O\left((1-z)^3\right)
\]
Again, since $\Phi_{1,12} = (1-z) + \frac{7}{10} (1-z)^2 + O\left((1-z)^3\right)$,
we have that
\[
    \Phi_{0,12} = -
      \frac{h_0}{2\pi i} \Phi_{1,11} - \frac{h_1}{2\pi i}
      \Phi_{1,12} -
      \left(\frac{h_2}{2\pi i} - \frac{7}{10} \frac{h_1}{2\pi i} \right) \Phi_{1,14}\\
\]
from which the second row of $M_{01}$ follows. The explicit values of
the coefficients $h_m$ are determined by~\eqref{eq:hm_n=4} as
  \begin{align}
  \label{eq:hm}
    h_0 &=
    \Gamma(\tfrac{1}{5})^2\Gamma(\tfrac{2}{5})^2\Gamma(\tfrac{4}{5})\sum_{\ell=0}^\infty
    \frac{(\frac{1}{5})_\ell(\frac{2}{5})_\ell}{(\frac{3}{5})_\ell
      \ell!}\Hyp{-\ell,\frac{3}{5},\frac{4}{5}}{1,1}{1}\\
    h_1 &=
    \tfrac{2}{15}\Gamma(\tfrac{1}{5})^2\Gamma(\tfrac{2}{5})^2\Gamma(\tfrac{4}{5})
    \sum_{\ell=0}^\infty
    \frac{(\frac{1}{5})_\ell(\frac{2}{5})_\ell}{(\frac{8}{5})_\ell
      \ell!}\Hyp{-\ell,\frac{3}{5},\frac{4}{5}}{1,1}{1}\notag\\
    h_2 &=
    \tfrac{7}{100}\Gamma(\tfrac{1}{5})^2\Gamma(\tfrac{2}{5})^2\Gamma(\tfrac{4}{5})\sum_{\ell=0}^\infty
    \frac{(\frac{1}{5})_\ell(\frac{2}{5})_\ell}{(\frac{13}{5})_\ell
      \ell!}\Hyp{-\ell,\frac{3}{5},\frac{4}{5}}{1,1}{1}\notag
  \end{align}

Moreover, Theorem~\ref{thm:22} (see Example~\ref{expl:29}) also yields
\[
G_3(z) = k_0 + k_1 (1-z) + k_2 (1-z)^2 + O((1-z)^3)
\]
Hence, by the same reasoning as above
\[
  \begin{aligned}
    \Phi_{0,13}
    &= \frac {5\,k_0}{(2\pi i)^2}  \Phi_{1,11} + \frac {5\,k_1}{(2\pi i)^2} \Phi_{1,12} + \left(\frac {5\,k_2}{(2\pi i)^2} -\frac{7}{10} \frac {5\,k_1}{(2\pi i)^2}\right) \Phi_{1,14} \\
  \end{aligned}
\]
from which the third row of $M_{01}$ follows. The explicit values of
the coefficients $k_m$ are determined by~\eqref{eq:km} as
  \begin{align}
    \label{eq:km2}
    k_0 &= \Gamma(\tfrac{2}{5}) \Gamma(\tfrac{3}{5})
    \Gamma(\tfrac{4}{5}) \sum_{\ell=0}^\infty
    \frac{\Gamma(\frac{1}{5}+\ell)^2 }{\Gamma(\ell+1)^2} \\
    &\phantom{=} \cdot \left(
      \e{-i\pi\frac{2}{5}}
      \frac{\Gamma(\frac{1}{5})\Gamma(\frac{2}{5})\Gamma(\frac{2}{5}+\ell)}{\Gamma(\frac{3}{5})\Gamma(\frac{3}{5}+\ell)}
      {}_3F_2\left(\genfrac{}{}{0pt}{0}{\frac{2}{5},\frac{2}{5}+\ell,\frac{2}{5}}{\frac{4}{5},\frac{3}{5}+\ell};1\right)
    \right. \notag\\
    &\phantom{=} \left. +\, \e{-i\pi\frac{3}{5}}
      \frac{\Gamma(-\frac{1}{5})\Gamma(\frac{3}{5})\Gamma(\frac{3}{5}+\ell)}{\Gamma(\frac{2}{5})\Gamma(\frac{4}{5}+\ell)}
      {}_3F_2\left(\genfrac{}{}{0pt}{0}{\frac{3}{5},\frac{3}{5}+\ell,\frac{3}{5}}{\frac{6}{5},\frac{4}{5}+\ell};1\right)
    \right)  \notag\\
    k_1 &= \Gamma(\tfrac{4}{5}) \Gamma(\tfrac{7}{5})
    \Gamma(\tfrac{8}{5}) \sum_{\ell=0}^\infty
    \frac{\Gamma(\frac{1}{5}+\ell)^2 }{\Gamma(\ell+1)^2} \notag\\
    &\phantom{=} \cdot \left(
      \e{-i\pi\frac{2}{5}}
      \frac{\Gamma(\frac{1}{5})\Gamma(\frac{2}{5})\Gamma(\frac{2}{5}+\ell)}{\Gamma(\frac{8}{5})\Gamma(\frac{3}{5}+\ell)}
      {}_3F_2\left(\genfrac{}{}{0pt}{0}{\frac{2}{5},\frac{2}{5}+\ell,-\frac{3}{5}}{\frac{4}{5},\frac{3}{5}+\ell};1\right)
    \right. \notag\\
    &\phantom{=} \left. +\, \e{-i\pi\frac{3}{5}}
      \frac{\Gamma(-\frac{1}{5})\Gamma(\frac{3}{5})\Gamma(\frac{3}{5}+\ell)}{\Gamma(\frac{7}{5})\Gamma(\frac{4}{5}+\ell)}
      {}_3F_2\left(\genfrac{}{}{0pt}{0}{\frac{3}{5},\frac{3}{5}+\ell,-\frac{2}{5}}{\frac{6}{5},\frac{4}{5}+\ell};1\right)
    \right) \notag\\
    k_2 &= \tfrac{1}{2}\Gamma(\tfrac{4}{5})\Gamma(\tfrac{12}{5})
      \Gamma(\tfrac{13}{5}) \sum_{\ell=0}^\infty
    \frac{\Gamma(\frac{1}{5}+\ell)^2 }{\Gamma(\ell+1)^2} \notag\\
    &\phantom{=} \cdot \left(
      \e{-i\pi\frac{2}{5}}
      \frac{\Gamma(\frac{1}{5})\Gamma(\frac{2}{5})\Gamma(\frac{2}{5}+\ell)}{\Gamma(\frac{13}{5})\Gamma(\frac{3}{5}+\ell)}
      {}_3F_2\left(\genfrac{}{}{0pt}{0}{\frac{2}{5},\frac{2}{5}+\ell,-\frac{8}{5}}{\frac{4}{5},\frac{3}{5}+\ell};1\right)
     \right.\notag\\
    &\phantom{=} \left. +\, \e{-i\pi\frac{3}{5}}
      \frac{\Gamma(-\frac{1}{5})\Gamma(\frac{3}{5})\Gamma(\frac{3}{5}+\ell)}{\Gamma(\frac{12}{5})\Gamma(\frac{4}{5}+\ell)}
      {}_3F_2\left(\genfrac{}{}{0pt}{0}{\frac{3}{5},\frac{3}{5}+\ell,-\frac{7}{5}}{\frac{6}{5},\frac{4}{5}+\ell};1\right)
    \right) \notag
  \end{align}

Finally, we have $\Phi_{1,12}(y) = \frac{\sqrt{5}}{4\pi^2} \xi_4(z)$. Proposition~\ref{prop:16} with
$ \psi(x) = \prod_{\nu=1}^4 \left(x-\omega^{-\nu}\right)$, $\omega^5=1$
yields $\psi(1) = 5, \psi'(1) = 10, \frac{\psi''(1)}{2!} = 10,
\frac{\psi'''(1)}{3!} = 5$, so
that 
\[\Phi_{1,12}(y) = \tfrac{5}{2\pi i} \left( y^*_4(z) -2\, y^*_3(z) + 2\,y^*_2(z) - y_1^*(z) \right).\] 
Lemma~\ref{lem:4}(2) then yields $\Phi_{0,14}(z) = 2\pi i
\,\Phi_{1,12}(y)$. This result has also been obtained through a
monodromy argument in~\cite{Candelas:1991rm}. 

\end{proof}

At present, we are not aware of any identities that help evaluating
the infinite sums in $l_0$, $w_m$, $h_m$ and $k_m$. Numerical
evaluation shows, however, that the following identity should hold:
\[
  \Im k_m = \pi i h_m, m=0,1,2.
\]

\bibliography{books,math,Commun_Math_Phys,KEK,Miscellaneous,Nuclear_Physics_B,Other_Journals,Hodge-CM,Hodge_Structures,Modular,other_math,new_math,Archive-9108,Archive-9203,Archive-9301,Archive-9306,Archive-9310,Archive-9312,Archive-9403,Archive-9407,Archive-9409,Archive-9412,Archive-9502,Archive-9506,Archive-9509,Archive-9511,Archive-9601,Archive-9603,Archive-9605,Archive-9607,Archive-9608,Archive-9609,Archive-9611,Archive-9612,Archive-9701,Archive-9702,Archive-9703,Archive-9704,Archive-9705,Archive-9706,Archive-9707,Archive-9709,Archive-9710,Archive-9711,Archive-9712,Archive-9801,Archive-9802,Archive-9803,Archive-9804,Archive-9805,Archive-9806,Archive-9807,Archive-9809,Archive-9810,Archive-9811,Archive-9812,Archive-9901,Archive-9902,Archive-9903,Archive-9904,Archive-9905,Archive-9906,Archive-9907,Archive-9908,Archive-9909,Archive-9910,Archive-9911,Archive-9912,Archive-0001,Archive-0002,Archive-0003,Archive-0004,Archive-0005,Archive-0006,Archive-0007,Archive-0008,Archive-0009,Archive-0010,Archive-0011,Archive-0012,Archive-0102,Archive-0103,Archive-0104,Archive-0105,Archive-0106,Archive-0108,Archive-0109,Archive-0110,Archive-0111,Archive-0112,Archive-0202,Archive-0203,Archive-0204,Archive-0205,Archive-0206,Archive-0207,Archive-0208,Archive-0210,Archive-0211,Archive-0301,Archive-0302,Archive-0303,Archive-0304,Archive-0305,Archive-0308,Archive-0309,Archive-0310,Archive-0311,Archive-0312,Archive-0401,Archive-0402,Archive-0403,Archive-0404,Archive-0405,Archive-0406,Archive-0407,Archive-0408,Archive-0409,Archive-0410,Archive-0411,Archive-0412,Archive-0502,Archive-0503,Archive-0504,Archive-0505,Archive-0506,Archive-0507,Archive-0511,Archive-0512,Archive-0605,Archive-0606,Archive-0607,Archive-0609,Archive-0610,Archive-0611,Archive-0612,Archive-0701,Archive-0702,Archive-0705,Archive-0706,Archive-0708,Archive-0709,Archive-0803,Archive-0804,Archive-0805,Archive-0808,Archive-0812,Archive-0901,Archive-0904,Archive-0910,Archive-1205,Archive-1211,Archive-1308}
\bibliographystyle{amsplain}

\end{document}